\documentclass[11pt]{article}

\usepackage{amsfonts,amsthm}
\usepackage{amsmath,color,enumerate}
\usepackage{graphicx,hyperref,cite,ulem,cancel}
\usepackage[a4paper,portrait]{geometry}

\newcommand{\beq}{\begin{equation}}
\newcommand{\eeq}{\end{equation}}
\newcommand{\bea}{\begin{eqnarray}}
\newcommand{\eea}{\end{eqnarray}}
\newcommand{\beas}{\begin{eqnarray*}}
\newcommand{\eeas}{\end{eqnarray*}}

\newtheorem{theorem}{Theorem}[section]
\newtheorem{definition}[theorem]{Definition}
\newtheorem{proposition}[theorem]{Proposition}

\newtheorem{lemma}[theorem]{Lemma}
\newtheorem{remark}[theorem]{Remark}
\newtheorem{example}[theorem]{Example}
\newtheorem{examples}[theorem]{Examples}
\newtheorem{assumption}[theorem]{Assumption}
\newtheorem{foo}[theorem]{Remarks}

\newcommand{\R}[1]{\mathbb{R}}     % absolute value

\title{Fractional stable random fields on the Sierpi\'nski gasket}

\author{Fabrice Baudoin\footnote{Partly funded by NSF DMS-2247117 and Simons Foundation}, C\'eline Lacaux}

\begin{document}

\maketitle

\begin{abstract}
We define and study  fractional stable random fields on the Sierpi\'nski gasket. Such fields are formally defined as $(-\Delta)^{-s} W_{K,\alpha}$, where $\Delta$ is the Laplace operator on the gasket and $W_{K,\alpha}$ is a stable random measure. Both Neumann and Dirichlet boundary conditions for $\Delta$ are considered. Sample paths regularity and scaling  properties are obtained.  The techniques we develop are general and extend to the more general setting of the Barlow fractional spaces.
\end{abstract}

\noindent
{\footnotesize{\it Keywords:} Fractional stable fields, fractional Riesz kernels, H\"older continuity.}\\ 
{\footnotesize{\it 2020 Mathematics Subject Classification.} Primary: 
60G52, 
60G17  
; Secondary: 
60G60,
    28A80, 
    31E05
}

\tableofcontents

\section{Introduction}

Fractional random fields are widely studied since the last decades, especially in the Euclidean framework that is when the field is indexed by $\mathbb R^d$.    The Fractional Brownian field $B_H=(B_H(x))_{x\in\mathbb R^d}$ with Hurst index $H\in (0,1)$, introduced in \cite{MVN,Kolmogorov}, is certainly the most popular fractional Gaussian field. It admits several representation, as harmonizable and moving average ones: this has led to several generalizations aiming to relax some  properties, such as self-similarity or isotropy,  but also allowed to introduce non-Gaussian fractional random fields. Being far from exhaustive, we mention  \cite{CM89,KT94,Cohen01,Cohen02} as examples of non-Gaussian Euclidean fractional random fields,   \cite{AyacheLeger02,Estrade03,OSSRF,ARX09} as examples of anisotropic fractional (Gaussian or not) fields  and \cite{BEJARO,PELLEVY,Lacaux02,BLS11-MOSRF,AEH18} as examples of  multifractional (Gaussian or not) fields. Recently, \cite{MR3466837} has also named  fractional Gaussian field $X$ a Gaussian field defined as solution, in a distribution sense of 
$$
X=(-\Delta)^{-s} W,
$$
where $s\ge 0$,  $W$ is a white noise and $\Delta$ is the Euclidean Laplacian on $\mathbb R^d$. This class extends Fractional Brownian motions, including smoother random fields especially and some other classical fields such as white noise, log-correlated Gaussian random field and the Gaussian free field. Moreover,  this approach has been considered before in \cite{BEJARO}  in terms of reproducing kernel to define elliptic Gaussian processes. More precisely, the reproducing kernel of an elliptic Gaussian process and then its covariance function is characterized by an elliptic symmetric positive pseudodifferential operator $A$ and in particular choosing for $A=(-\Delta)^s$ leads to the particular case of fractional Euclidean Gaussian fields.    \\

Moreover, since several years, the interest in random fields with heavy-tails, such as $\alpha$-stable random fields, is increasing and many papers focus on their sample path regularity. Most often in the literature, the fractional Euclidean fields studied are defined via an integral representation, such as moving average or harmonizable one. As far as we know, \cite{BenassiRoux} is the only work that has considered a class of stable random fields defined via an elliptic operator and then followed the approach of \cite{BEJARO}. For a Gaussian field ($\alpha=2$), it is well-known that the smoothness is linked to the behavior of the second order moment 
$$
\mathbb E \left(X(x)-X(y)\right)^2
$$
as $x\to y$. Based on this behavior and on a chaining argument,  the entropy method allows to obtain quite fine controls of the modulus of continuity for Gaussian random fields (see e.g. \cite{adler2009}). When $\alpha<2$, the $\alpha$-stable random field is not anymore square integrable and then the methods to study its smoothness properties differ from the Gaussian case. For example, \cite{Xiao10} has obtained a uniform modulus of continuity of stable random fields using a chaining argument and  the notion of maximal moment index. {Another} technique proposed in \cite{kono1970} and followed by \cite{MR2531090,BLS11-MOSRF}, relies on a chaining argument and a representation as a conditionally Gaussian LePage random series.  Considering a more general framework,  \cite{BL-2015-Bernoulli} has also given some sufficient  conditions to derive an upper bound for the modulus of continuity of  Euclidean stable random fields. Other studies, such as \cite{Ayache1,Ayache2,Ayache3,BenassiRoux}, are based on a wavelet series representation of the random field under study.   \\

Non-Euclidean fractional fields have also been considered quite recently:  \cite{Istas1} has defined the fractional Brownian field  on manifolds and metric spaces giving its covariance function, replacing the Euclidean norm by a distance $d$ and \cite{IL13} has considered some moving average random fields. Using the fractional Riesz kernels and the Laplace Beltrami operator  and following the approach of \cite{BEJARO,BenassiRoux,MR3466837},  \cite{Gelbaum14} has then considered fractional Brownian fields over manifolds.  The previous works \cite{BL22-FGF,BC22-DFGF} on fractional fields indexed by a fractal set have also followed this approach. \\

The main goal to this paper is to extend \cite{BL22-FGF,BC22-DFGF} to the stable framework, that is to define and study 
\begin{equation}
\label{FGF-def-intro}
X_{s,\alpha}=(-\Delta)^{-s} W_{K,\alpha},
\end{equation}
where $\alpha\in (0,2]$, $s\ge 0$, $\Delta$ is the Neummann or Dirichlet Laplace operator on the Sierpi\'nski gasket $K$ and $W_{K,\alpha}$ is a symmetric real-valued $\alpha$-stable random measure with control measure the Hausdorff measure $\mu$ on $K$. 
As in \cite{BL22-FGF,BC22-DFGF},  the definition \eqref{FGF-def-intro} has to be interpreted in a distribution sense, that is as 
$$
X_{s,\alpha}(f)=W_{K,\alpha}\left((-\Delta)^{-s}f\right)=\int_{K}
(-\Delta)^{-s} f \, dW_{K,\alpha}
$$
for $f$ in a suitable Sobolev space $H^{-s}(K)$ of functions.   
Compared to \cite{BL22-FGF,BC22-DFGF},  the white noise  on $L^2(K,\mu)$, that is the Gaussian random measure with intensity the measure $\mu$, is here replaced by an $\alpha$-stable measure $W_{K,\alpha}$. As a consequence, $X_{s,\alpha}$ is an $\alpha$-stable symmetric random field, that is for any functions $f_1,\ldots,f_n$ and any real numbers $\lambda_1,\ldots,\lambda_n$, $\sum_{i=1}^n \lambda_i X_{s,\alpha}(f_i)$ is a symmetric $\alpha$-stable random variable. Note that  when $\alpha=2$,  $W_{K,\alpha}=W_{K,2}$ is nothing else than a white noise on $L^2(K,\mu)$ 
and hence $X_{s,2}$  is or the Neumman fractional Gaussian field studied in \cite{BL22-FGF} or the Dirichlet one studied in \cite{BC22-DFGF}. 
In this paper, we then focus on the case $\alpha<2$, for which the techniques for studying the sample path regularity differ from the ones used in the Gaussian case ($\alpha=2$).  \\

 As in \cite{BL22-FGF,BC22-DFGF}, we are interested in the existence of a random field $\widetilde{X}_{s,\alpha}$ defined pointwise on $K$ which is a density of $X_{s,\alpha}$, that is which satisfies   for $f$ in a suitable space $\mathcal{S}(K)$ of test functions, 
$$
X_{s,\alpha}(f)=\int_{K} (-\Delta)^{-s} f(x) \widetilde{X}_{s,\alpha}(x)\,\mu(dx)\quad a.s. 
$$
Then we establish, see Theorem  \ref{Existence-Density} that such density exists if and only if  $s>s_\alpha$  which extends Proposition 2.17 of \cite{BC22-DFGF} to the $\alpha$-stable framework. 
The proof relies on a stochastic Fubini theorem for stable random fields (see \cite{taqqu}). 
More or less, up to an additive random variable, this density field  admits as integral representation 
$$
 \widetilde{X}_{s,\alpha}(x)= \int_{K} G_s(x,y) W_{K,\alpha}(dy),
$$
with $G_s$ the fractional Riesz kernel,  
and it is itself an $\alpha$-stable random field. When $\alpha=2$, it corresponds to the Gaussian density field defined in  \cite{BL22-FGF,BC22-DFGF}. Moreover, according to \cite{Rosinski89}, roughly speaking, the sample paths of $\widetilde{X}_{s,\alpha}$ can not be smoother than the kernel $G_s$: this leads to unboundedness of $\widetilde{X}_{s,\alpha}$ for $s\in (s_\alpha, s_0)$, where $s_0$ does not depend on the  index stability $\alpha\in (0,2)$ (see Theorems \ref{SPR} and \ref{SPR2}), whereas in the Gaussian framework $\alpha=2$, the random field $\widetilde{X}_{s,\alpha}$ has always h\"olderian sample paths. In the Euclidean framework, this difference is also observed: for example, moving average fractional stable motion indexed by $\mathbb R^d$ can be unbounded, see e.g.\cite{taqqu,KM91,T89,MR2531090}, and especially  are always not continuous when ${d\ge 2}$. Moreover, if $s>s_0$,   we then establish that the fractional field $\widetilde{X}_{s,\alpha}$ admits a continuous modification and obtain an upper bound of its modulus of continuity. As in \cite{MR2531090,BLS11-MOSRF,BL-2015-Bernoulli}, our proof relies on its representation as a conditionally Gaussian LePage random series but slightly differs since  {another} key ingredient is the Garsia-Rumsey-Rodemich inequality. This inequality has first been established in \cite{GRR70} for Euclidean processes indexed by $\mathbb R$, extended  to random fields indexed by a metric space in \cite{AS96} and proved for random fields indexed by a fractal in \cite{BP88}. Compared to \cite{MR2531090,BLS11-MOSRF,BL-2015-Bernoulli}, it allows us to avoid the construction of a specific dense sequence $\mathcal D\subset K$ on which the increments of $\widetilde{X}_{s,\alpha}$ are controlled.  \\

To conclude, let us remark that fractional stable random fields could be defined and studied in the framework of the Barlow fractional spaces as in Section 4 of \cite{BL22-FGF}. Indeed our general analysis of such fields is based on heat kernel estimates which are available in this general framework. However, for  the sake of presentation we focused on the specific example of the Sierpi\'nski gasket. In that example further specific properties, like scaling and symmetries can then  be proved, see Section 3.5. \\

The paper is organized as follows. In Section 2, after some background on the Sierpi\'nski gasket and Dirichlet form on it, the Neumann Fractional Laplacian is introduced as in \cite{BL22-FGF} but for any parameter $s\ge 0$. The fractional Riesz kernels are also studied. Section 3 is then devoted to the definition and study of  Neumann fractional stable random fields and Section 4 extends the results to Dirichlet fractional stable random fields. Appendix A is devoted to the proof of a technical lemma about the convolution of two Riesz kernels.

\section{Sierpi\'nski Gasket and Neumann Fractional Riesz kernels}

\subsection{Definition of the gasket}

Let us recall the definition of the (standard) Sierpi\'nski gasket, denoted by $K$ in this paper and  of its Hausdorff measure $\mu$. We identify $\mathbb{R}^2 \simeq \mathbb C$ and consider the three vertices $q_0=0$, $q_1=1$ and $q_2=e^{\frac{i\pi}{3}}$. The Sierpi\'nski gasket can then be defined inductively or  characterized owing  the contraction maps 
\[
F_i(z)=\frac{1}{2}(z-q_i)+q_i, \quad i\in \{0,1,2\} 
\]
as follows. 
 For further details we refer to the book by Kigami \cite{kigami}.
\begin{definition}
The Sierpi\'nski gasket is the unique non-empty compact set $K \subset \mathbb C$ such that 
\[
K=\bigcup_{i=0}^2 F_i (K).
\]
\end{definition}

\begin{figure}[htb]\centering
 	\includegraphics[trim={60 10 180 60},height=0.25\textwidth]{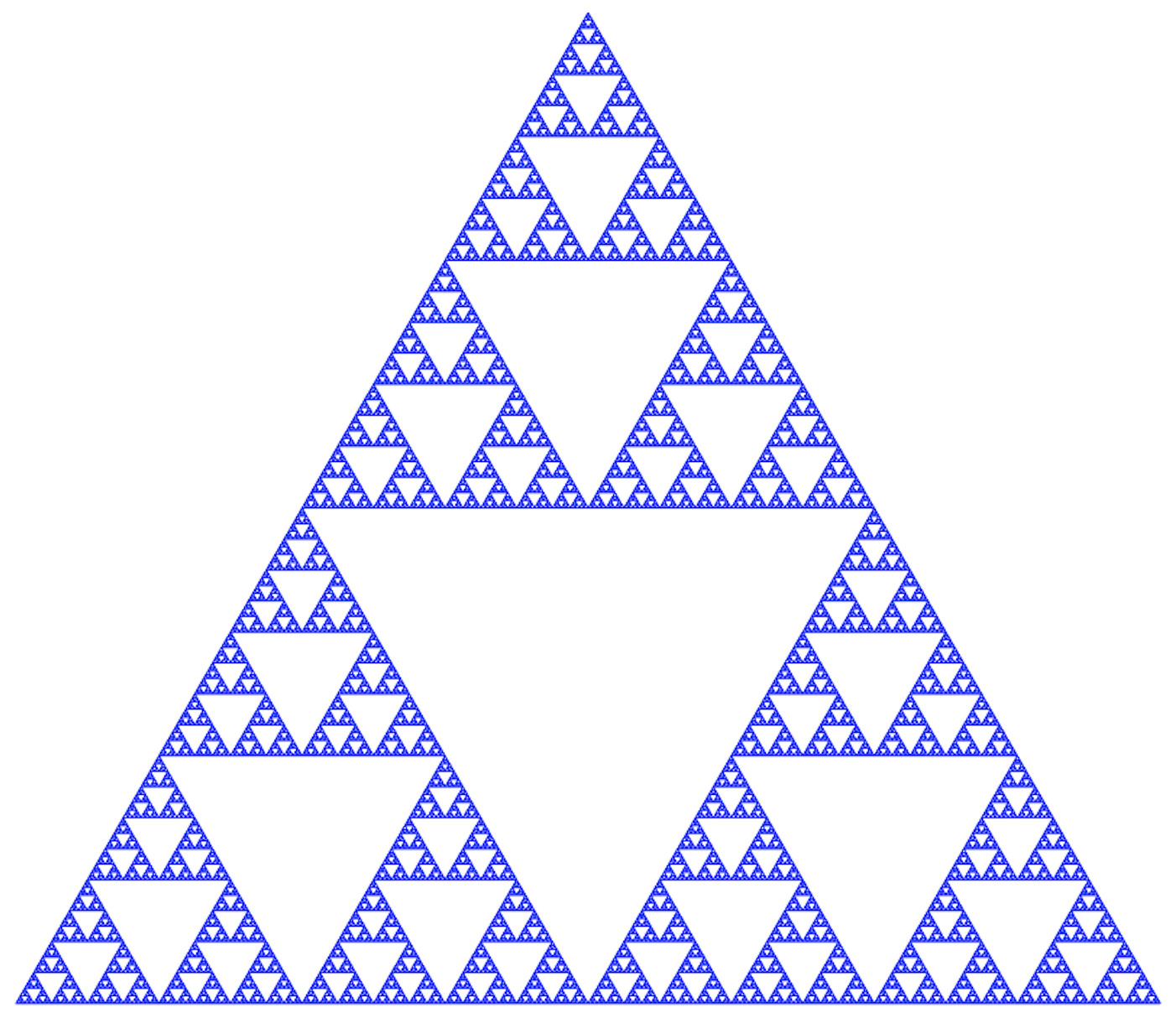}
 	\caption{Sierpi\'nski gasket.} \label{fig-SG}
 \end{figure}

In this paper, $K$ is equipped with the Euclidean metric denoted by $d$ (so that its diameter is ${\rm diam} K=1$), and of its Borel $\sigma$-field $\mathcal{K}$. Then, there exists on $(K,\mathcal{K})$   a unique probability $\mu$, named normalized Hausdorff measure,   such that for every $i_1, \cdots, i_n \in \{ 0,1,2 \} $
\[
\mu \left(  F_{i_1} \circ \cdots \circ F_{i_n}  (K)\right)=3^{-n}.
\]
As in \cite{BL22-FGF,BC22-DFGF}, one key property in our study is the  Alhfors property satisfied by the measure $\mu$. So let us recall that denoting by 
$$
d_h=\frac{\ln 3}{\ln 2}
$$
the Hausdorff dimension of $K$, the continuous probability $\mu$ is  $d_h$-Ahlfors regular, that is there exist constants $c,C>0$ such that 
\begin{equation}
\label{Ahlfors}\forall r\in (0,1], \ 
c r^{d_h} \le \mu (B(x,r)) \le C r^{d_h},
\end{equation}
where  
$B(x,r)$ is the metric ball of $K$ with center $x$ and radius~$r$. The values of $c$ and $C$ are useless in our study, but for the reader interested in,  \cite{Bar98} has proved that Equation \eqref{Ahlfors} holds e.g. with $c=1/3$ and $C=18$.

\subsection{Canonical Dirichlet form and heat kernel estimates}
Following \cite{kigami}, let us introduce in this {section} the canonical Dirichlet form  the Neumann Laplacian is associated with considering a graph approximation of the gasket. Let  $V_0= \{ q_0,q_1,q_2 \}$,    
and  for any $\boldsymbol{i}=(i_1,\ldots,i_n)\in \{0,1,2\}^n$,  
$$
W_{\boldsymbol{i}}=F_{i_1} \circ \cdots \circ F_{i_n} (V_0)$$
and then consider 
$$V_n=\bigcup_{\boldsymbol{i}\in \{0,1,2\}^n} W_{\boldsymbol{i}}  \quad \textrm{and} \quad 
V_* =\bigcup_{ n \ge 0} V_n .
$$

For any $n\in \mathbb N,$ let us now define the quadratic form $\mathcal{E}_n$ on $\mathbb{R}^{V_n}$ by 

\[{\forall f \in \mathbb{R}^{V_n}, \quad 
\mathcal{E}_n (f,f)=\frac{1}{2} \left( \frac{5}{3} \right)^n \sum_{\boldsymbol{i}\in\{0,1,2\}^n} \sum_{x,y\in W_{\boldsymbol{i}}} \left(f(x)-f(y)\right)^2 .}
\]
Then, setting 
\[
\mathcal{F}_* = \left\{ f \in \mathbb{R}^{V_*}, \lim_{n \to \infty}   \mathcal{E}_n \left(f,f\right) <+\infty \right\} 
\]
for any function $f\in \mathcal{F}_*$, let 
\begin{align}\label{dirichlet limit}
\mathcal{E} (f,f) =\lim_{n \to \infty}   \mathcal{E}_n (f,f).
\end{align}
In addition, any function $f \in \mathcal{F}_*$ can uniquely be extended into a continuous function defined on the whole~$K$. Denoting by $\mathcal F$ the set of such extensions,  $(\mathcal{E},\mathcal{F})$ is then a local regular Dirichlet form on $L^2(K,\mu)$ such that for every $f,g \in \mathcal F$
\[
\mathcal{E} (f,g)=\frac{5}{3} \sum_{i=0}^2 \mathcal{E} (f \circ F_i , g \circ  F_i ),
\]
 see the book by Kigami \cite{kigami}.  Moreover, the semigroup $\{P_t\}_{t\ge 0}$ associated with $\mathcal{E}$ is stochastically complete (i.e. $P_t 1=1$) and, from \cite{BP88},   has a jointly continuous heat kernel $p_t(x,y)$ with respect to the reference measure $\mu$. Finally,  
 there exist some constants $c_{1},c_{2}, c_3, c_4 \in(0,\infty)$ such that  for every  $(x,y)\in K \times K$ and $t\in\bigl(0,1)$
 \begin{equation}\label{eq:subGauss-upper}
 c_{1}t^{-d_{h}/d_{w}}\exp\biggl(-c_{2}\Bigl(\frac{d(x,y)^{d_{w}}}{t}\Bigr)^{\frac{1}{d_{w}-1}}\biggr) 
 \le p_{t}(x,y)\leq c_{3}t^{-d_{h}/d_{w}}\exp\biggl(-c_{4}\Bigl(\frac{d(x,y)^{d_{w}}}{t}\Bigr)^{\frac{1}{d_{w}-1}}\biggr)
 \end{equation}
 where we recall that  $d_h=\frac{\ln 3}{\ln 2}$  is the Hausdorff dimension of $K$ and where the parameter $d_w=\frac{\ln 5}{\ln 2}$  is called the walk dimension. Since $d_w > 2$,  the heat kernel  estimates in Equation \eqref{eq:subGauss-upper}  are said to be sub-Gaussian  estimates.

Let $\Delta$ be the generator of the Dirichlet form $\mathcal E$. We refer to $\Delta$ as the Neumann Laplacian on $K$. This terminology comes from the fact that functions in the domain of $\Delta$ satisfy a Neumann boundary condition, see Theorem 7.3.9 in \cite{kigami}.

\subsection{Neumann Fractional Laplacian}

In this section, as in \cite{BL22-FGF}, we introduce  the Neumann fractional  Laplacian $(-\Delta)^{-s}$   on 
\[
L^2_0(K,\mu)=\left\{ f \in L^2(K,\mu), \int_K f d\mu=0 \right\}. 
\]
Note that \cite{BL22-FGF} assumes that $s>d_h/(2d_w)$, but following  \cite{BC22-DFGF} which deals with the Dirichlet Laplacian, the definition of the Neumann fractional  Laplacian $(-\Delta)^{-s}$ is here easily extended to any $s\ge 0$. To do so, let us consider $0<\lambda_1\le \lambda_2\le   \cdots \le  \lambda_n \le \cdots$  the  eigenvalues of $-\Delta$ on $L^2_0(K,\mu)$ and an orthonormal basis $\Phi_j \in L_0^2(K,\mu) $, $j\ge 1$ such that 
$$
\Delta \Phi_j=-\lambda_j \Phi_j,
$$
see \cite{kigami} for details. In addition, from \cite{kigami},  we have $\Phi_j \in \mathrm{dom} (\Delta) \subset \mathcal F$, $j\ge 1$, and 
\begin{align}\label{spectral}
p_t(x,y)=1+\sum_{j=1}^{+\infty} e^{-\lambda_j t} \Phi_j(x) \Phi_j(y)
\end{align} 
where the series converges uniformly on any compact set $(t, x,y)\in [a,b]\times K \times K$ with $0<a<b$. 

Let us now consider $f\in L^2_0(K,\mu)$, so that in $L^2_0(K,\mu)$, 
\begin{align}\label{eg: decomposition f}
f=\sum_{j=1}^{+\infty}\int_{K}f(x)\Phi_j(x) \mu (dx) \,\Phi_j.  
\end{align}
Then, for any $s\ge 0$, 
$$
\left| \lambda_{j}^{-s} \int_{K}f(x)\Phi_j(x) \mu (dx) \right|^2\le \lambda_{1}^{-2s}
\left|  \int_{K}f(x)\Phi_j(x) \mu (dx) \right|^2
$$
and so the series 
$$
\sum_{j=1}^{+\infty}  \lambda_{j}^{-s} \int_{K}f(x)\Phi_j(x) \mu (dx) \,\Phi_j
$$
converges in $L^2(K,\mu)$ and also in $L_0^2(K,\mu)$  by continuity of the scalar product on $L^2(K,\mu)$. 
Then, according to the previous lines, the Neumann  fractional Laplacian of order $-s$ is well-defined as follows and is a bounded operator from $L^2_0(K,\mu)$ to itself. 

\begin{definition}[Neumann Fractional Laplacians]\label{FLaplacian}
Let $s \ge 0$. For $f \in L_0^2(K,\mu)$, the Neumann fractional Laplacian $(-\Delta)^{-s}$ on $f$ is defined as
\begin{align}\label{def laplace fraction}
(-\Delta)^{-s} f  =\sum_{j=1}^{+\infty} \lambda_j^{-s} \int_K  \Phi_j(y) f(y) \mu(dy) \, \Phi_j.
\end{align}
Note that  $(-\Delta)^{-s} f \in L_0^2(K,\mu)$ and $\left\| (-\Delta)^{-s} f\right\|_{L^2(K,\mu)}\le \lambda_{1}^{-s}\left\| f\right\|_{L^2(K,\mu)}$.
\end{definition}

\begin{remark}
The definition \eqref{def laplace fraction} is consistent with classical functional calculus where, for any function $h$,  $h(-\Delta)$ is defined for a function $f$ as
\[
h(-\Delta) f  =\sum_{j=1}^{+\infty} h(\lambda_j) \int_K  \Phi_j(y) f(y) \mu(dy) \, \Phi_j.
\]
The intuition is that $h(-\Delta) $ is defined in such a way that $h(-\Delta)\Phi_j=h(-\lambda_j) \Phi_j$ and then extended by linearity using \eqref{eg: decomposition f}. In our case, $h(x)=x^{-s}$. We refer to \cite[VII.1]{{Reed}} for further details about functional calculus.
\end{remark}

\begin{remark} The definition of the Neumann fractional Laplacian $(-\Delta)^{-s}$ on $f$ does not depend on the choice of the orthonormal basis $\Phi_j$, $j\ge 1$. Actually, considering $0<\lambda_{i_1}<\lambda_{i_2}<\ldots <\lambda_{i_n} <\cdots $ the distinct eigenvalues of $-\Delta$ on $L_0^2(K,\mu)$, the Neumann fractional Laplacian $(-\Delta)^{-s}$ can be rewritten as 
$$
(-\Delta)^{-s} f=\sum_{j=1}^{+\infty} \lambda_{i_j}^{-s}\, p_{E_{j}}(f)
$$
where $p_{E_j}$ is the orthogonal projection on the eigenspace  $E_j\subset L^2_0(K,\mu)$  associated with the  eigenvalue $\lambda_{i_j}$. 
\end{remark}

\begin{remark}
It follows from \eqref{def laplace fraction} that for $f \in L^2_0(K,\mu)$ and $s,t \ge 0$,
\begin{align}\label{composition laplace frac}
(-\Delta)^{-s-t} f= (-\Delta)^{-s}(-\Delta)^{-t} f.
\end{align}
\end{remark}

  We then consider  the following space of test functions
\[
\mathcal{S}(K)= \left\{ f \in C_0(K),  \forall k \ge 0 \lim_{n \to +\infty} n^k \left| \int_K  \Phi_n(y) f(y) \mu(dy) \right| =0 \right\},
\]
where
\[
C_0(K)=\left\{ f \in C(K), \int_K f d\mu=0 \right\}\subset L_0^2(K,\mu).
\]
 As in \cite{BC22-DFGF},  when $f \in \mathcal{S}(K)$, then $f \in \bigcap_{k \in \mathbb N}  \mathrm{dom}((-\Delta)^{k})$ and so  for every $k \in \mathbb N$, $(-\Delta)^{k} f \in C(K)$ and is H\"older continuous, see \cite{BP88}.  We then endow $\mathcal{S}(K)$ with the topology defined by the family of norms
\[
\| f \|_k = \| (-\Delta)^{k} f \|_{L^\infty(K,\mu)}, \quad k  \in \mathbb N.
\]
In addition, as in \cite{BC22-DFGF}, by Lemma 5.1.3 in \cite{kigami}, one can check that
$$
\sum_{j=1}^{+\infty} \lambda_j^{-s} <+\infty \iff s > \frac{d_h}{d_w} 
$$
and then that the space $\mathcal{S}(K)$ is a  nuclear space. Then its dual space $\mathcal{S}'(K)$ for the latter topology  is the analogue of the space of tempered distributions in our setting.

We also note that for $f \in \mathcal{S}(K)$ and $s \ge 0$, it follows from {Lemma 5.1.3} and Theorem 4.5.4 in \cite{kigami} (see also Lemma 4.8 in \cite{BC22-DFGF}) that the series 
\[
\sum_{j=1}^{+\infty} \lambda_j^{-s} \int_K  \Phi_j(y) f(y) \mu(dy) \, \Phi_j
\]
is uniformly convergent on $K$, {so that for $f\in \mathcal{S}(K)$, $(-\Delta)^{-s} f$ is defined pointwise on $K$ by \eqref{def laplace fraction} and $(-\Delta)^{-s} f\in \mathcal{S}(K)$.}

\subsection{Neumann Fractional Riesz kernels}

Let us recall the definition of a fractional Riesz kernel on $K$ for the Neumann Laplacian, see \cite{BL22-FGF}. 
\begin{definition}
For a parameter $s>0$, we define the fractional Riesz kernel $G_s $ by 
\begin{align}\label{green function}
G_s(x,y)=\frac{1}{\Gamma(s)} \int_0^{+\infty} t^{s-1} (p_t(x,y) -1)dt, \quad x,y \in K, \, x\neq y,
\end{align}
with $\Gamma$ the gamma function. Moreover, $G_s(x,x)$ is also well-defined when $s>d_h/d_w$ (see Lemma 2.4 in \cite{BL22-FGF}). 
\end{definition}
Note that when $s\in (0,d_h/d_w]$,  the function $y\mapsto G_s(x,y)$ is well-defined $\mu$-almost everywhere, since $\mu\left(\{x\}\right)=0$. Moreover, since the heat kernel $p_t$ is symmetric, for any $s>0$  the Riesz fractional kernel $G_s$ is also a symmetric function. \\

The relationship between the fractional Riesz kernel $G_s$ and the fractional Laplacian defined in the previous section is given by the following representation formula which  can be proved as in \cite{BC22-DFGF}. 

\begin{lemma}\label{riesz-kernel laplace}
 Let $s > 0$. For $f \in \mathcal{S}(K)$, and $x \in K$
\begin{equation}\label{eq:Fractional Laplacian-Riesz Kernel}
 (-\Delta)^{-s} f (x) = \int_K G_{s}(x,y) f(y) \mu(dy).
\end{equation}   
\end{lemma}

The fractional Riesz kernel $G_s$ has already been studied in \cite{BL22-FGF}. First, Proposition 2.6 in \cite{BL22-FGF}  gives a control, that allows to study the  square integrability of $G_s(x,\cdot)$ needed in the Gaussian framework. In the stable framework, we are now wondering when $G_s(x,\cdot)\in L^\alpha(K,\mu)$, so that its stochastic integral with respect to the stable random measure $W_{K,\alpha}$ is well-defined. This integrability property is the key point to define pointwise fractional stable fields in next section. Proposition 2.6 in \cite{BL22-FGF} already allows to state a sufficient condition but need to be {completed}  in order to characterize completely the  integrability of $G_s(x,\cdot)$. In addition Theorem 2.10  in \cite{BL22-FGF} is interested in the H\"older property in $L^2(K,\mu)$ fulfilled by $x\mapsto G_s(x,\cdot)$. Especially, when studying fractional Gaussian random fields, this implies, using classical entropy methods, a control of their modulus of continuity.  However, in the framework of  fractional stable fields, it is  not sufficient: in our study, a {\it pointwise} H\"older property for $G_s$ is needed. Roughly speaking the stable random field under study can not be smoother than the fractional Riesz kernel involved in its definition.\\

Let us first complete  Proposition 2.6 in \cite{BL22-FGF}.    
\begin{proposition}\label{Control-G}
\hfill
\begin{enumerate}
\item 
If $s\in (0,d_h/d_w)$, then there exist three positive constants $C_1,C_2,C_3$ such that for any $x,y\in K$, $x\ne y$,
$$
C_1 d(x,y)^{sd_w-d_h} -C_2\le G_s(x,y) \le C_3 d(x,y)^{sd_w-d_h} .
$$
\item Let $s=d_h/d_w$. Then there exist three positive constants $C_1,C_2,C_3$ such that for any $x,y\in K$ such that $x\ne y$, 
$$
-C_1 \ln d(x,y) -C_2\le G_s(x,y) \le C_3 \max(\left| \ln d(x,y)\right|,1).
$$
\end{enumerate} 
\end{proposition}

\begin{remark} As a consequence, for $s\le d_h/d_w$, $G_s$ is unbounded.  This will imply  unboundedness of  the  fractional stable random field associated with $G_s$, see Theorem \ref{SPR} in the sequel.
\end{remark}
\begin{proof}  
The upper bounds 
have already been established in \cite{BL22-FGF}, see Proposition~2.6. So let us now consider $s\in (0,d_h/d_w]$ and prove the lower bounds. 
Let $x,y\in K$, $x\ne y$. As in \cite{BL22-FGF}, we have 
$$
G_s(x,y)= G_{s}^{(1)}(x,y)+G_{s}^{(2)}(x,y)
$$
where 
$$
\left\{\begin{array}{rcl}
\displaystyle G_{s}^{(1)}(x,y) &=&\displaystyle\frac{1}{\Gamma(s)} \int_0^{1} t^{s-1} (p_t(x,y) -1)dt = \frac{1}{\Gamma(s)} \int_0^{1} t^{s-1} p_t(x,y) dt-\frac{1}{\Gamma(s+1)} \\[15pt]
\displaystyle G_{s}^{(2)}(x,y)& = &\displaystyle\frac{1}{\Gamma(s)} \int_1^{+\infty} t^{s-1} (p_t(x,y) -1)dt.
\end{array}
\right. 
$$
By the heat kernel lower bound \eqref{eq:subGauss-upper}, 
$$
G_{s}^{(1)}(x,y) \ge \frac{c_1}{\Gamma(s)} \int_0^{1} t^{s-1-d_h/d_w} \exp\left(-c_{2}\left(\frac{d(x,y)^{d_{w}}}{t}\right)^{\frac{1}{d_{w}-1}}\right)dt-\frac{1}{\Gamma(s+1)} 
$$
Then using the change of variable $t= u d(x,y)^{d_{w}}$, we have
$$
G_{s}^{(1)}(x,y) \ge\displaystyle  \frac{c_1d(x,y)^{sd_w-d_h} }{\Gamma(s)}\int_0^{1/d(x,y)^{d_{w}}}  u^{s-1-d_{h}/d_{w}} \exp\left(-c_2 u^{\frac{1}{1-d_{w}}}\right)  du-\frac{1}{\Gamma(s+1)}.
$$ 
Then,  denoting by $C_1$ a  positive constant (depending only on $s,d_h,d_w
$) whose value may change at each line,  when $s=d_h/d_w$, since $d(x,y)\le 1$, 
$$\begin{array}{rcl}
G_{s}^{(1)}(x,y) &\ge& \displaystyle  C_1  \int_{1
}^{1/d(x,y)^{d_{w}}}   u^{-1} \exp\left(-c_2 u^{\frac{1}{1-d_{w}}}\right)  du-\frac{1}{\Gamma(s+1)}\\[15pt]
&\ge&\displaystyle C_1 \int_{1
}^{1/d(x,y)^{d_{w}}}   u^{-1}   du-\frac{1}{\Gamma(s+1)}\\[15pt]
\end{array}
$$
since $c_2>0$ and $d_w>1$. Then, 
$$\begin{array}{rcl}
G_{s}^{(1)}(x,y) 
&\ge &\displaystyle-C_1 \ln d(x,y) 
-\frac{1}{\Gamma(s+1)}.
\end{array}
$$
Now  when $s<d_h/d_w$, since $d(x,y)\le 1$,  
$$G_{s}^{(1)}(x,y) \ge\displaystyle  C_1d(x,y)^{sd_w-d_h}   -\frac{1}{\Gamma(s+1)},
$$
with 
$$
C_1=\frac{c_1}{\Gamma(s)}\int_0^{1}
u^{s-1-d_{h}/d_{w}} \exp\left(-c_{2}{u}^{\frac{1}{1-d_{w}}}\right) du
$$
a finite positive constant since $c_2>0$ and $d_w>1$.  
Moreover, as noticed in \cite{BL22-FGF}, $G_{s}^{(2)}$ is uniformly bounded on $K\times K$, so that there exists a positive constant $c$ such that  for any $x,y\in K$
$$
G_{s}^{(2)}(x,y)\ge -c
$$
which concludes the proof choosing $C_2=\frac{1}{\Gamma(s+1)}+c$.  
\end{proof}

 Let us now study the integrability properties of the fractional Riesz kernels. 
\begin{lemma}\label{Integrabilite-G} Let $p \in (0,+\infty)$. Then for any $x\in K$, $y \mapsto G_s(x,y)$ is in $L^p(K,\mu)$ iff $s>\max\left(\frac{(p-1)d_h}{p d_w},0\right)$. Moreover, for $s>\max\left(\frac{(p-1)d_h}{p d_w},0\right)$,
$$\sup_{x\in K} \int_{K} \left|G_s(x,y)\right|^p \mu(dy)<+\infty.
$$
\end{lemma}

This is a consequence of the following lemma,  of Proposition \ref{Control-G} and of Proposition 2.6 in \cite{BL22-FGF}.

\begin{lemma} \label{Integrabilite-dist}  Let $\gamma \in\mathbb R$. 
Then for any $x\in K$, 
$$
\int_K\frac{ \mu(dy)}{d(x,y)^\gamma} <+\infty \iff \gamma < d_h.
$$
 Moreover,  for any $\gamma <d_h$,  
$$
\sup_{x\in K} \int_K\frac{ \mu(dy)}{d(x,y)^\gamma} <+\infty.
$$

\end{lemma}
\begin{proof}  By the proof of Proposition 2.7 in \cite{BL22-FGF}, if $ \gamma < d_h$, one has 
$$\forall x\in K, \ 
\int_K\frac{ \mu(dy)}{d(x,y)^\gamma} <+\infty
$$
and also 
$$
\sup_{x\in K} \int_K\frac{ \mu(dy)}{d(x,y)^\gamma} <+\infty.
$$
Let us now consider $\gamma \ge d_h$. Then by compactness of $K$, there exists a  positive constant $C$ such that 
\[\forall x,y\in K,\ 
d(x,y)^{\gamma}\le C d(x,y)^{d_h},
\]
and
it is therefore enough to prove that for any $x\in K$, 
\[
\int_K\frac{ \mu(dy)}{d(x,y)^{d_h}}=+\infty.
\]
We 
use   a dyadic  annuli decomposition as follows. 
Since ${\rm diam}\, K=1$,  one has for  every $N \ge 1$ and for any $x\in K$,  
\begin{align*}
 \int_K \frac{\mu(dy)}{d(x,y)^{d_h}} 
 & \ge \sum_{j=0}^{N} \int_{B\left(x, 2^{-j}\right)\setminus B\left(x, 2^{-j-1}\right) } \frac{\mu(dy)}{d(x,y)^{d_h}} \\
 & \ge   \sum_{j=0}^{N} 2^{ j d_h} \mu \left( B\left(x, 2^{-j}\right)\setminus B\left(x, 2^{-j-1}\right) \right) \\
 & \ge    \sum_{j=0}^{N} 2^{ j d_h}\left(  \mu \left( B\left(x, 2^{-j}\right) \right)-\mu \left( B\left(x, 2^{-j-1} \right)\right) \right)\\
  & \ge  \sum_{j=0}^{N} 2^{ j d_h}  \mu \left( B\left(x, 2^{-j}\right) \right)-\sum_{j=0}^{N} 2^{ j d_h}\mu \left( B\left(x, 2^{-j-1}\right) \right) \\
  &\ge  \mu\left(B\left(x,1\right)\right)-2^{Nd_h}\mu \left( B\left(x, 2^{-N-1} \right)\right)+(1-2^{-d_h})\sum_{j=1}^{N} 2^{ j d_h}  \mu \left( B\left(x, 2^{-j}\right) \right).
\end{align*}

By the Ahlfors regularity \eqref{Ahlfors} of the measure $\mu$, for any $x\in K$,  when $N \to +\infty$,
\[
\sum_{j=1}^{N} 2^{ j d_h}  \mu \left( B(x, 2^{-j}) \right) \to +\infty
\]
whereas $2^{Nd_h}\mu \left( B(x, 2^{-N-1} \right))$ remains bounded. Therefore for any $x\in K$, 
\[
\int_K\frac{ \mu(dy)}{d(x,y)^{d_h}}=+\infty,
\]
which concludes the proof.
\end{proof}

We are now ready to prove Lemma \ref{Integrabilite-G}. 

\begin{proof}[Proof of Lemma \ref{Integrabilite-G}] Let us first assume that $s\in (0,d_h/d_w)$. Then, since $\mu$ is a probability, by Proposition~\ref{Control-G} and Lemma~\ref{Integrabilite-dist}, 
$$\begin{array}{rcl}
\displaystyle \int_{K}\left|G_s(x,y)\right|^{p}\mu(dy) <+\infty 
&\iff&  \displaystyle\int_K\frac{\mu(dy)}{d(x,y)^{p\left(d_h-sd_w\right)}}  <+\infty \\[15pt]&\iff &\displaystyle p\left(d_h-sd_w\right)<d_h\\[15pt]
&\iff & s>\frac{(p-1)d_h}{p d_w}.
\end{array}
$$
Moreover, for $\max\left(0,\frac{(p-1)d_h}{pd_w}\right)<s< d_h/d_w$,  by Proposition \ref{Control-G} and Lemma \ref{Integrabilite-dist}, there exists a  positive constant $C$ such that 
$$
\sup_{x\in K}\int_{K}\left|G_s(x,y)\right|^{p}\mu(dy) \le C\sup_{x\in K}\displaystyle\int_K\frac{\mu(dy)}{d(x,y)^{p\left(d_h-sd_w\right)}}  <+\infty. 
$$
Now for $s\ge d_h/d_w$,  using Proposition 2.6  in \cite{BL22-FGF} and compactness of $K$, for any $\gamma \in (0,d_h/p)$, one sees that there exists a  positive constant $C$ (depending on $\gamma)$ such that for any $x,y\in K$, $x\ne y$
$$
\left|G_s(x,y)\right| \le \frac{C}{d(x,y)^\gamma}.
$$
Then, since $p \gamma <d_h$,  by Lemma \ref{Integrabilite-dist},
$$
\sup_{x\in K}\int_{K}\left|G_s(x,y)\right|^{p}\mu(dy) \le C\sup_{x\in K}\displaystyle\int_K\frac{\mu(dy)}{d(x,y)^{p\gamma}}  <+\infty,$$
which concludes the proof. 
\end{proof}

\vskip 5pt

To conclude this section, next theorem is interested in the smoothness of  the Riesz kernel $G_s$, that is in the H\"older property it satisfies. This property is the key property to study sample paths properties of fractional stable random fields, which will be defined as an integral of the kernel $G_s$ with respect to a stable random measure (see Theorems \ref{Existence-Density} and  \ref{SPR} in next section). Next theorem is only interested in the case where $s>d_h/d_w$ since otherwise, Proposition \ref{Control-G} implies that the function $G_s$ is unbounded.
\begin{theorem}\label{Holder-Riesz}\hfill

\begin{enumerate}
\item  For any $s \in \left( \frac{d_h}{d_w} , 1\right)$, there exists a constant $C>0$ such  that for every  $x,y,z \in K$,
\[
| G_s (x,z)-G_s(y,z)|  \le Cd(x,y)^{ sd_w-d_h} .
\]
\item  For $s \ge 1$, there exists a  constant $C>0$ such  that for every  $x,y,z \in K$,
\[
| G_s (x,z)-G_s(y,z)|  \le Cd(x,y)^{ d_w-d_h} \max\left(\left|\ln d(x,y)\right|,1 \right).
\]
where  $0^{ d_w-d_h} \max\left(\left|\ln 0\right|,1 \right)=0$ by convention.
\end{enumerate}
\end{theorem}

\begin{remark}

In the range $s \in \left( \frac{d_h}{d_w} , 1\right)$, it is reasonable to conjecture that the exponent $sd_w-d_h$ is optimal. Indeed, by its definition $x\to G_s(x,z)$ is in the domain of the operator $(-\Delta)^s$ and, as in Section 9 of \cite{BP88}, one can reasonably conjecture that the only functions in the domain of the operator $(-\Delta)^s$ which are $sd_w-d_h+\varepsilon$ H\"older continuous are the constant functions. We also refer to \cite{BST99} for further comments on the regularity of functions associated with the Laplacian. 
\end{remark}

\begin{proof}  Let $x,y,z\in K$. Without loss of restriction, we can assume $x\ne y$. 

\begin{enumerate} 
\item Let us first assume that $s\in (d_h/d_w,1)$. Then, 
for $\delta \in (0,1]$
$$
G_s(x,z)-G_s(y,z)=\frac{1}{\Gamma(s)}\int_{0}^{+\infty} t^{s-1} \left(p_t(x,z)-p_t(y,z) \right)dt=I_1(x,y,z)+I_2(x,y,z)
$$
where 
$$
\left\{\begin{array}{rcl}
\displaystyle I_1(x,y,z)&=&\displaystyle \frac{1}{\Gamma(s)}\int_{0}^{\delta} t^{s-1} \left(p_t(x,z)-p_t(y,z) \right)dt\\[15pt]
\displaystyle I_2(x,y,z)&=&\displaystyle\frac{1}{\Gamma(s)}\int_{\delta}^{+\infty} t^{s-1} \left(p_t(x,z)-p_t(y,z) \right)dt.\\[15pt]
\end{array}
\right. 
$$

In the following, $C$ is a generic  positive constant, independent from $(\delta,x,y,z)$, which value may change at each line. 

First applying the triangle inequality and the upper bound heat kernel \eqref{eq:subGauss-upper}, for any $t\in (0,1)$, 
$$
\left|p_t(x,z)-p_t(y,z) \right| \le C t^{-d_h/d_w} 
$$
so that 
$$\begin{array}{rcl}
\displaystyle \left|I_1(x,y,z)\right|&\le& C\displaystyle \int_{0}^{\delta} t^{s-1-d_h/d_w} dt
\\[15pt]
&\le & C\delta^{s-d_h/d_w}
\end{array}
$$
since $s-d_h/d_w>0$. 

Moreover, by  Lemma 3.4 in \cite{ABCRST3}, for any $t>0$
\begin{equation}
\label{pt-Holder-property}
\left|p_t(x,z)-p_t(y,z) \right| \le \frac{C d(x,y)^{d_w-d_h}}{t} 
\end{equation}
so that 
$$
\begin{array}{rcl}
\displaystyle \left|I_2(x,y,z)\right|&\le&  \displaystyle  Cd(x,y)^{d_w-d_h} \int_{\delta}^{+\infty} t^{s-2} dt
\\[15pt]
&\le& Cd(x,y)^{d_w-d_h} \delta^{s-1}
\end{array}
$$
since $s<1$. 

Hence, by combining the obtained upper bounds for $I_1$ and $I_2$ and choosing $\delta=d(x,y)^{d_w}\in (0,1]$, 
$$
\begin{array}{rcl}
\displaystyle \left|G_s(x,z)-G_s(y,z)\right|
&\le& Cd(x,y)^{sd_w-d_h} .
\end{array}
$$

\item \begin{enumerate}
\item We first establish Assertion 2. for $s=1$. 
Then, for $\delta\in(0,1]$
$$
G_1(x,z)-G_1(y,z)=\int_{0}^{+\infty}  \left(p_t(x,z)-p_t(y,z) \right)dt=I_1(x,y,z)+I_2(x,y,z)+I_3(x,y,z)
$$
where 
$$
\left\{\begin{array}{rcl}
\displaystyle I_1(x,y,z)&=&\displaystyle \int_{0}^{\delta}  \left(p_t(x,z)-p_t(y,z) \right)dt\\[15pt]
\displaystyle I_2(x,y,z)&=&\displaystyle\int_{\delta}^{1/\delta}  \left(p_t(x,z)-p_t(y,z) \right)dt\\[15pt]
\displaystyle I_3(x,y,z)&=&\displaystyle\int_{1/\delta}^{+\infty} \left(p_t(x,z)-p_t(y,z) \right)dt.
\end{array}
\right. 
$$
In the following, as previously, $C$ is a generic  positive constant, independent from $(\delta,x,y,z)$, which value may change at each line. 

First the same arguments as for $s<1$ lead to 
$$\begin{array}{rcl}
\displaystyle \left|I_1(x,y,z)\right|&\le&  C\delta^{1-d_h/d_w}
.\end{array}
$$
Moreover, by applying \eqref{pt-Holder-property},
$$
\begin{array}{rcl}
\displaystyle \left|I_2(x,y,z)\right|&\le&  \displaystyle  Cd(x,y)^{d_w-d_h} \int_{\delta}^{1/\delta} \frac{dt}{t}
\\[15pt]
&\le& -Cd(x,y)^{d_w-d_h} \ln \delta.
\end{array}
$$
In addition, applying the triangle inequality and Lemma 2.3 in \cite{BL22-FGF}, for $t\ge 1$
$$
\left|p_t(x,z)-p_t(y,z) \right|\le C \exp\left(-\lambda_1 t \right)
$$
where $\lambda_1>0$ is the lower eigenvalue of $-\Delta$. Hence, 
$$\begin{array}{rcl}
\displaystyle \left|I_3(x,y,z)\right|&\le& C \displaystyle \int_{1/\delta}^{+\infty}  \exp\left(-\lambda_1 t \right) dt
\\[15pt]
&\le & C\exp\left(-\frac{\lambda_1}{\delta}\right)
\\[15pt]
&\le &  C \delta^{1-d_h/d_w}.
\end{array}
$$
Therefore, combining the previous upper bounds and choosing $\delta =d(x,y)^{d_w} \in (0,1]$, 
$$\begin{array}{rcl}
\displaystyle
\left|G_1(x,z)-G_1(y,z)\right|&\le & C d(x,y)^{d_w-d_h}\left(1+   \left| \ln {d(x,y)}\right|\right)\\[15pt]
&\le & C d(x,y)^{d_w-d_h} \max\left(\left|\ln {d(x,y)}\right|,1\right).
\end{array}
$$

\item Let us now assume that $s>1$. Then, by the  Lemma \ref{conv-G}, 
$$
G_s(x,z)-G_s(y,z)=\int_{K} \left(G_1(x,u)-G_1(y,u)\right)G_{s-1}(u,z)\,\mu(du). 
$$
Hence, since $G_1$ satisfies Assertion 2., 
$$
\begin{array}{rcl}
\displaystyle
\left|G_s(x,z)-G_s(y,z)\right|&\le & \displaystyle C d(x,y)^{d_w-d_h} \max\left(\left|\ln {d(x,y)}\right|,1\right) \int_{K} \left|G_{s-1}(u,z)\right|\mu(du).
\end{array}
$$
 Moreover, since $s-1>0$, by Lemma \ref{Integrabilite-G} and symmetry of $G_{s-1}$, 
$$
\sup_{z\in K} \int_{K} \left|G_{s-1}(u,z)\right|\mu(du)=\sup_{z\in K} \int_{K} \left|G_{s-1}(z,u)\right|\mu(du) <+\infty,
$$
 which concludes the proof.
\end{enumerate}

\end{enumerate}
\end{proof}

We conclude the section with the following lemma.

\begin{lemma}\label{conv-G}
For any $s>0$ and any $t>0$, 
$$
G_{s+t}(x,y)=\int_K G_s(x,u)G_{t}(u,y)\mu(du)
$$
for any $x,y\in K$ such that $x\ne y$. This equation also holds for $x=y$ when $s+t>d_h/d_w$. 
\end{lemma}
\begin{proof} 
See Appendix \ref{conv-G-proof}.

\end{proof}

\section{Neumann Fractional Stable  random fields}

\subsection{Background on stable random measures and LePage Series}

Given a probability space  $\left(\Omega,\mathcal{F},\mathbb{P}\right)$ and $\alpha\in (0,2]$, we consider, on the measurable space $(K,\mathcal{K})$,    
a  real-valued symmetric stable random measure $W_{K,\alpha}$  
with control measure $\mu$.  Let us recall the definition of such a random measure following  \cite[Section 3.3]{taqqu}. We first consider 
$$
\mathcal{K}_0=\{A\in \mathcal{K},\, \mu(A)<+\infty\}.
$$
Then,  $W_{K,\alpha}$ is such that the three following properties hold:   
\begin{itemize}
	\item the random measure $W_{K,\alpha}$ is $\sigma$-additive, i.e. for any sequence  $(A_n)_{n\in\mathbb{N}} \in \mathcal{K}_0^\mathbb{N}$ of pairwise disjoint measurable sets such that $\bigcup_{n\in\mathbb N} A_n\in  \mathcal{K}_0$, 
 $$
 W_{K,\alpha}\left(\bigcup_{n\in\mathbb N} A_n\right)=\sum_{n\in\mathbb N}W_{K,\alpha}\left(A_n\right)\ \textrm{almost surely};
 $$
	\item $W_{K,\alpha}$ is independently scattered, that is for any sequence $(A_n)_{n\in\mathbb{N}} \in \mathcal{K}_0^\mathbb{N}$ of pairwise disjoint measurable sets, the random variables $W_{K,\alpha}(A_n)$, $n\in\mathbb{N}$, are independent;
 \item  for any $A\in\mathcal{K}_0$,   $W_{K,\alpha}(A)$ is a   symmetric $\alpha$-stable real-valued random variable  with scale parameter $\left(\mu(A)\right)^{1/\alpha}$, that is its characteristic function is given by
	$$
\mathbb{E}\left(\textup{e}^{iu W_{K,\alpha}(A)}\right)=\textup{e}^{- \left|u\right|^\alpha {\mu(A)}}, \quad u \in \mathbb R. 
	$$
\end{itemize}
Especially, if $\alpha=2$, then $W_{K,\alpha}$ is a real-valued Gaussian random measure. \\

Then for any real-valued $f\in L^\alpha(K,\mu)$, the stochastic integral 
$$
W_{K,\alpha}(f)=\int_{K} f(x)  \,W_{K,\alpha}(dx)
$$
is well-defined and is a symmetric $\alpha$-stable real-valued  random  variable with scale parameter 
$$
\|f\|_{L^\alpha(K,\mu)}=\left(\int_{K}\left|f(x)\right|^\alpha \,\mu(dx)\right)^{1/\alpha}
$$
so that, its characteristic function is then given by 
	$$
\mathbb{E}\left(\textup{e}^{iu W_{K,\alpha}(f)}\right)=\textup{e}^{- \left|u\right|^\alpha \|f\|_{L^\alpha(K,\mu)}^\alpha}, \quad u \in \mathbb R.
	$$ We refer to 
	\cite[Section 3.4]{taqqu}
	 for details on the construction of the stochastic integral $W_{K,\alpha}(f)$ and on its properties. \\
 
 In addition, when $\alpha\in (0,2)$, the stochastic integral $W_{K,\alpha}(f)$ can also be represented as a series, and such a series can be chosen to consider $W_{K,\alpha}(f)$ as a conditionally Gaussian random variable. This representation will be a key point to study the sample paths smoothness of the  fractional stable fields considered in this paper. Let us then now give  this LePage series representation. To do so, let us assume that the following assumption holds.

 \begin{assumption}
     \label{AssumptionLePage}
  $\left(T_n\right)_{n\ge 1}$, $\left(g_n\right)_{n\ge 1}$ and $\left(\xi_n\right)_{n\ge 1}$ are three  independent sequences of random variables such that 
 
 \begin{itemize}
 \item $T_n$ is the $n$th arrival time of a Poisson process with intensity 1;
 \item $g_n$, $n\ge 1$, are {i.i.d.} real-valued centered  Gaussian random variables;
 \item $\xi_n$, $n\ge 1$, are {i.i.d.} real-valued random variables with common distribution $\mu$. 
 \end{itemize}
 \end{assumption}
Under Assumption \ref{AssumptionLePage} and for $\alpha \in (0,2)$,  according to Equation (3.11.2) in~\cite[Section 3.9.1]{taqqu}, 
for any function $f\in L^\alpha(K,\mu)$, the series 
 \begin{equation}
 \label{LePageSeries}
 S(f)=D_\alpha\sum_{n=1}^{+\infty} T_n^{-{1/\alpha}}
 f\left(\xi_n\right)g_n ,
 \end{equation}
 where 
 \begin{equation}
     \label{csteCalpha}
D_\alpha=\left(\mathbb{E}\left(\left|g_1\right|^\alpha\right)\int_0^{+\infty} x^{-\alpha}\sin(x)\, dx\right)^{-{1}/{\alpha}} ,
 \end{equation}
 converges almost surely and 
\begin{equation}
 \label{SeriesVA}
 W_{K,\alpha}(f)\stackrel{(d)}{=} S(f).
 \end{equation}

Note also that, conditionally to $(T_n,\xi_n)_{n\ge 1}$, $S(f)$ is a centered real-valued  Gaussian random variable with variance
$$
\sigma^2(f)=D_\alpha^2\, \mathbb{E}\left(g_1^2\right) \,\sum_{n=1}^{+\infty} T_n^{-{2/\alpha}} 
\left|f\left(\xi_n\right)\right|^2. 
$$

\subsection{Definition of the Neumann Fractional Stable Field}

We first introduce the Neumman fractional stable field as a random variable which takes its values in  the space $\mathcal{S}'(K)$ of tempered distributions associated with the space $\mathcal{S}(K)$ of test functions. Next theorem, which states its existence,  extends Theorem 2.12 of \cite{BC22-DFGF} considering the Neumann Laplacian and stable random variables.

\begin{theorem}
Let $s>0$ and $\alpha \in (0,2]$. There exists a probability measure $\nu_{s,\alpha}$ on $\mathcal{S}'(K)$ such that for every $f \in \mathcal{S}(K)$
\[
\int_{\mathcal{S}'(K)} e^{i X(f)} \nu_{s,\alpha} (dX)=\textup{e}^{-  \|(-\Delta )^{-s} f\|_{L^\alpha(K,\mu)}^\alpha}.
\]
\end{theorem}

\begin{proof} 
By the Bochner-Minlos theorem in nuclear spaces (see \cite{MR0187273}, Theorem A) it is enough to prove that the functional
\[
\varphi: f \to \exp \left(-\int_K | (-\Delta)^{-s} f |^\alpha d\mu\right),
\]
which is defined on $\mathcal{S}(K)$ since $(-\Delta )^{-s} f\in \mathcal{S}(K)\subset\mathcal{C}(K)$ for $f\in \mathcal{S}(K)$,   is continuous  at $0$ and positive definite.  Let us first prove it is positive definite. Let $a_1,\cdots,a_n \in \mathbb{C}$ and $f_1,\cdots,f_n \in \mathcal{S}(K)$. We have 
\begin{align*}
\sum_{\ell,j=1}^n a_\ell \overline{a_j}\varphi(f_\ell-f_j)
 =& \sum_{\ell,j=1}^n a_\ell \overline{a_j} \exp \left(-\int_K | (-\Delta)^{-s} f_\ell - (-\Delta)^{-s} f_j |^\alpha d\mu\right) \\
 =&\sum_{\ell,j=1}^n a_\ell \overline{a_j} \mathbb{E} \left( \exp \left( iW_{K,\alpha} \left[ (-\Delta)^{-s} f_\ell-(-\Delta)^{-s} f_j)  \right]\right)\right)\\
 =&\sum_{\ell,j=1}^n a_\ell \overline{a_j}\mathbb{E} \left( \exp \left( iW_{K,\alpha} \left[ (-\Delta)^{-s} f_\ell \right]\right)\exp \left(- iW_{K,\alpha} \left[ (-\Delta)^{-s} f_j\right]\right)\right) \\
 =&{\mathbb{E}\left(\left| \sum_{\ell=1}^n a_\ell \exp \left( iW_{K,\alpha} \left[ (-\Delta)^{-s} f_\ell \right]\right)\ \right|^2\right)} \ge 0.
\end{align*}
Next, we prove that $\varphi$ is continuous at 0.  Since $\mu$ is a probability measure and $0<\alpha \le 2$, for every $f \in \mathcal{S}(K)\subset L^2(K,\mu)$, 
 \[
 \left(\int_K | (-\Delta)^{-s} f |^\alpha d\mu\right)^{2/\alpha} \le \int_K | (-\Delta)^{-s} f |^2 d\mu \le \lambda_1^{-2s} \int_K f^2 d\mu
 \]
 where the last inequality follows from Definition \ref{FLaplacian}. 
 Since the convergence in $\mathcal{S}(K)$ implies the convergence in $L^2(K,\mu)$, we conclude that $\varphi$ is indeed continuous at $0$.
\end{proof}

The previous result yields for $s > 0$ and $\alpha \in (0,2]$ the existence of a random variable $X_{s,\alpha}$  on $\mathcal{S}'(K)$ such that for every $f \in \mathcal{S}(K)$
\[
\mathbb{E} \left( e^{iX_{s,\alpha} (f)}\right)=\textup{e}^{-  \|(-\Delta )^{-s} f\|_{L^\alpha(K,\mu)}^\alpha}.
\]
On the other hand, since $(-\Delta )^{-s}:\mathcal{S}(K) \to \mathcal{S}(K)$, one can define the operator $(-\Delta )^{-s}:\mathcal{S}'(K) \to \mathcal{S}'(K)$  thanks to the duality formula
\[
((-\Delta )^{-s} X)(f)=X((-\Delta )^{-s} f), \quad X \in \mathcal{S}'(K), f \in \mathcal{S}(K).
\]
In distribution, we have then
\[
(-\Delta )^{-s'} X_{s,\alpha}=X_{s+s',\alpha}.
\]
 Despite the previous nice interpretations as random variables in $\mathcal{S}'(K)$, in this paper we will work with the following \textit{concrete} and more constructive realization of the fractional stable random field as a stochastic integral; we a priori lose the a.s. continuity in the $\mathcal{S}(K)$ topology but this does not play any role in the sequel.

 Similarly to  \cite{BC22-DFGF} we consider the Sobolev space $H^{-s}(K)$ which is defined as the completion of $\mathcal{S}(K)$ with respect  to the norm $\| f \|_{H^{-s}(K)}=\| (-\Delta)^{-s} f \|_{L^2(K,\mu)}$. This Sobolev space does not depend on the choice of the basis since it is also the completion of $L^2_0(K,\mu)$ with respect to that norm.

\begin{definition}\label{def fractional field} 
  Let $s > 0$ and $\alpha \in (0,2]$. The  Neumann fractional $\alpha$-stable field $\left\{ X_{s,\alpha}(f), f\in H^{-s}(K) \right\}$ is defined as
  \begin{equation}\label{repre stable def}
X_{s,\alpha}(f) =W_{K,\alpha} ((-\Delta)^{-s} f)=\int_K (-\Delta)^{-s} f (z) \,W_{K,\alpha}(dz).
  \end{equation}
\end{definition}
Note  that the field $X_{s,\alpha}$ is well-defined since for $f \in H^{-s}(K)$, $(-\Delta)^{-s} f \in L^2 (K,\mu)\subset L^\alpha(K,\mu)$.  

\begin{remark}
It follows from Definition \ref{def fractional field} that for every $f \in H^{-s}(K)$,
\[
\mathbb{E} \left( \textup{e}^{i X_{s,\alpha}(f) }\right)=\textup{e}^{-  \|(-\Delta )^{-s} f\|_{L^\alpha(K,\mu)}^\alpha}.
\]
As in \cite{BC22-DFGF} which deals with the case $\alpha=2$, this property could be used as a definition of a Neumann fractional $\alpha$-stable field. However, it will be more convenient here to work with the concrete representation \eqref{repre stable def}.
\end{remark}

\subsection{Density field}

As in \cite{BL22-FGF,BC22-DFGF}, let us now study the existence of a density with respect to the Hausdorff measure $\mu$ of the Neumann fractional stable field $X_{s,\alpha}$. We only consider the case $\alpha<2$ since the case $\alpha=2$ has already been considered in \cite{BL22-FGF}, even though next result is more precise (see Remark \ref{Comp} below). 

\begin{theorem}\label{Existence-Density}
  Let $\alpha\in (0,2)$. 
  
  \begin{enumerate}
\item   If    $s>\max\left(\frac{(\alpha-1)d_h}{\alpha d_w},0\right)$, then 
for every $ f\in \mathcal{S}(K)$, we have a.s.
  \[
\int_K f(x) \widetilde{X}_{s,\alpha}(x) \mu(dx) =X_{s,\alpha}(f)
  \]
  with $\widetilde{X}_{s,\alpha}=\left(\widetilde{X}_{s,\alpha}(x)\right)_{x\in K}$ the fractional $\alpha$-stable field  well-defined by 
\begin{equation}
\label{dens def}
\widetilde{X}_{s,\alpha}(x)=\int_K  G_s(x,z) W_{K,\alpha}(dz), \ x\in K. 
\end{equation}
\item Moreover, if there exists a  random field $(Y(x))_{x \in K}$ such that a.s
  \[
\int_K |Y(x)| \mu(dx) <+\infty 
  \]
 and such that for every $ f\in \mathcal{S}(K)$ we have a.s.
 \[
\int_K f(x) Y(x) \mu(dx) =X_{s,\alpha}(f),
  \]
  then  we have $s>\max\left(\frac{(\alpha-1)d_h}{\alpha d_w},0\right)$ and for a.e. $x\in K$  with probability one $Y (x)=\widetilde{X}_{s,\alpha}(x)+\int_K Y(z) \mu(dz)$.
\end{enumerate}
\end{theorem}

\begin{remark}\label{Comp}\hfill

\begin{enumerate}
\item When $\alpha=2$, the random field $\widetilde{X}_{s,\alpha}$ is still  well-defined by \eqref{dens def} when $s>\frac{d_h}{2d_w}$ and,  up to a multiplicative constant, is the fractional Brownian motion of index $H=sd_\omega-\frac{d_h}{2}$ introduced in \cite{BL22-FGF}. Moreover according  {to} Theorem~1.1 in \cite{BL22-FGF}, it is also the density field of $X_{s,2}$ when $s<1-\frac{d_h}{2d_w}$, and this can still be proved for any $s>\frac{d_h}{2d_w}$ as done in  \cite{BC22-DFGF} for the Dirichlet Laplacian.   
\item In the Euclidean framework, that is when we considered the Laplacian on $\mathbb R^d$,  the heat kernel  is given by 
$$
p_t(x,y)=\frac{1}{(4\pi t)^{d/2}}\exp \left(-\frac{\left\|x-y\right\|^2}{4t}\right),\quad t>0, \,x\in \mathbb R^d, \,y\in \mathbb R^d.
$$
In this case, due to the behavior of the heat kernel $p_t$,  the Riesz kernel $G_s$ cannot be defined by Equation~\eqref{green function} for any real $s$. Following e.g. \cite{Gelbaum14} which deals with fractional Brownian fields (so $\alpha=2$) indexed by manifolds, we can define it by 
$$
G_s(x,y)=\frac{1}{\Gamma(s)}\int_{0}^{+\infty} t^{s-1}\left(p_t(x,y)-p_t(0,y)\right)\,dt , \quad x\ne y,
$$
as soon as $s\in (-\infty,d/2+1)$. In addition, for $x\ne y$ and $y\ne 0$, 
$$
G_s(x,y)=\left\{\begin{array}{ll}
b_s\left(\|x-y\|^{2s-d}-\|y\|^{2s-d}\right)&\textrm{ when $s\ne \frac{d}{2}$}\\[10pt]
b_s\left(\log \|y\| -\log \|x-y\|\right) &\textrm{ when $s= \frac{d}{2}$}
\end{array}
\right.
$$
with 
$$
b_s=\left\{\begin{array}{ll}\displaystyle \frac{\Gamma\left(\frac{d}{2}+1-s\right)}{4^s\pi^{\frac{d}{2}}\left(\frac{d}{2}-s\right)\Gamma(s)}&\textrm{ when $s\ne \frac{d}{2}$}\\[15pt]
\displaystyle\frac{2}{\left(4\pi\right)^{\frac{d}{2}}\Gamma(s)}&\textrm{ when $s= \frac{d}{2}$}.
\end{array}
\right.
$$
 Then replacing in Equation \eqref{dens def}, $K$ by $\mathbb R^d$ and $W_{K,\alpha}$ by $W_\alpha$ a real symmetric $\alpha$-stable random measure with control measure the Lebesgue measure $\mu$,     the random field $\widetilde{X}_{s,\alpha}$ is well-defined when $s\in \left(\frac{d}{2}-\frac{d}{2\alpha},\frac{d}{2}-\frac{d}{2\alpha}+\frac{1}{2}\right)$. Note that in this case,  $\widetilde{X}_{s,\alpha}(0)=0$. Moreover, when $\alpha=2$, the random field $\widetilde{X}_{s,\alpha}$ is the fractional Brownian field of Hurst index $H=2s-d/2$; when $\alpha<2$, it is 
a fractional moving average stable  field of index $H=2s-d+\frac{d}{\alpha}\ne \frac{d}{\alpha}$ when $s\ne d/2$ and a log-fractional stable field when $s=d/2$, see e.g. \cite{taqqu}. 
\end{enumerate}
\end{remark}

Before turning to the proof we start with a preliminary lemma.

\begin{lemma}\label{Integrabilite-Dens}
    Let $\alpha\in (0,2)$ and  $s>\max\left(\frac{(\alpha-1)d_h}{\alpha d_w},0\right)$. Then $\widetilde{X}_{s,\alpha}$ is well-defined by \eqref{dens def} and a.s.
 \[
\int_K \left|\widetilde{X}_{s,\alpha}(x)\right| \mu(dx) <+\infty 
  \]  
  and
 \[
\int_K \widetilde{X}_{s,\alpha}(x) \mu(dx)=0. 
  \]  
\end{lemma}
\begin{proof} Let $s>\max\left(\frac{(\alpha-1)d_h}{\alpha d_w},0\right)$. 
Then by Lemma \ref{Integrabilite-G}, the Riesz kernel $z\mapsto G_s(x,z)$ is in $L^{\alpha}(K,\mu)$ and so $\widetilde{X}_{s,\alpha}$ is well-defined by \eqref{dens def}. 

Let us now assume that  $\alpha >1$.   By Lemma \ref{Integrabilite-G} applied with $p=\alpha$, 
$$
C_1=\sup_{x\in K} \int_K\left|G_s(x,y)\right|^\alpha \mu(dy) <+\infty
$$
and then
$$
\int_{K}\left(\int_K\left|G_s(x,y)\right|^\alpha \mu(dy)\right)^{1/\alpha}\!\!\mu(dx)<C_1^{1/\alpha}\mu(K)=C_1^{1/\alpha}<+\infty. 
$$
Hence, by Theorem 11.3.2 in \cite{taqqu}, since $\mu$ is a $\sigma$-finite measure on $K$ and since $\alpha>1$, a.s. 
$$
\int_K \left|\widetilde{X}_{s,\alpha}(x)\right| \mu(dx) <+\infty .
$$

Let us now assume that $\alpha \in (0,1]$. Then $\max\left(\frac{(\alpha-1)d_h}{\alpha d_w},0\right)=0$ and so $s>0$. We can choose $\varepsilon>0$ (depending on $s$) small enough such that $s>\frac{\varepsilon d_h}{(1+\varepsilon)d_w }>0$.  Then applying Lemma \ref{Integrabilite-G} with $p=1+\varepsilon$,
$$
C_2=\sup_{x\in K} \int_K\left|G_s(x,y)\right|^{1+\varepsilon}\mu(dy) <+\infty. 
$$
Moreover by symmetry of $G_s$,
$$
C_2=\sup_{y\in K} \int_K\left|G_s(x,y)\right|^{1+\varepsilon} \mu(dx)
$$
and then
$$
\int_{K}\left(\int_K\left|G_s(x,y)\right|^{1+\varepsilon} \mu(dx)\right)^{\alpha/(1+\varepsilon)}\mu(dy)\le C_2^{\alpha/(1+\varepsilon)}\mu(K)=C_2^{\alpha/(1+\varepsilon)}<+\infty.
$$
 Hence by Theorem 11.3.2 in \cite{taqqu}, since $\mu$ is a $\sigma$-finite measure on $K$ and since $1+\varepsilon >\alpha$, a.s. 
$$
\int_K \left|\widetilde{X}_{s,\alpha}(x)\right|^{1+\varepsilon} \mu(dx) <+\infty .
$$
This implies that a.s. 
$$
\int_K \left|\widetilde{X}_{s,\alpha}(x)\right|\mu(dx) <+\infty
$$
since $\mu$ is a probability.  Then, applying Theorem 11.4.1 in \cite{taqqu},  that is  a stochastic Fubini theorem for stable random fields, it follows that a.s.  
$$
\int_K \widetilde{X}_{s,\alpha}(x)\mu(dx)= \int_{K}\left(\int_K G_s(x,z)\mu(dx)\right)W_{K,\alpha}(dz)=0
$$
since 
for every $z \in K$
\[
\int_K G_s(x,z)\mu(dx)=0.
\]
\end{proof}

\begin{proof}[{Proof of Theorem \ref{Existence-Density}}]
Let $f\in \mathcal{S}(K)$. Then $f$ is continuous on $K$ and  bounded on the compact $K$. Hence 
$$
\int_K \left| f(x) \widetilde{X}_{s,\alpha}(x) \right|\mu (dx) \le C \int_K \left|  \widetilde{X}_{s,\alpha}(x) \right|\mu (dx) <+\infty \quad \textrm{a.s.}
$$
by Lemma \ref{Integrabilite-Dens}. Therefore applying Theorem 11.4.1 in \cite{taqqu}, that is a stochastic Fubini theorem, and Lemma~\ref{riesz-kernel laplace}, almost surely,
\begin{align*}
    \int_K f(x) \widetilde{X}_{s,\alpha}(x) \mu(dx)
   =&\int_K f(x) \left( \int_K G_s(x,y)W_{K,\alpha}(dy)\right) \mu(dx) \\
   =&\int_K \left( \int_K G_s(x,y) f(x) \mu(dx) \right)W_{K,\alpha}(dy)\\
   =&\int_K (-\Delta)^{-s} f(y) W_{K,\alpha}(dy).
 \end{align*}

Next, assume that there exists a random field $(Y(x))_{x \in K}$ such that a.s
  \[
\int_K |Y(x)| \mu(dx) <+\infty 
  \]
 and such that for every $ f\in \mathcal{S}(K)$ we have a.s.
 \[
\int_K f(x) Y(x) \mu(dx) =X_{s,\alpha}(f). 
  \]
  We can assume $\int_K  Y(x) \mu(dx)=0$, otherwise we could consider $Y(x)-\int_K  Y(x) \mu(dx)$.
  Then, we have
  \[
\mathbb{E}\left(e^{iu\int_K f(x) Y(x) \mu(dx)} \right)=\textup{e}^{- \left|u\right|^\alpha \|(-\Delta)^{-s}f\|_{L^\alpha(K,\mu)}^\alpha}.
  \]
  We use this equality with $f=p_t(y,\cdot)-1 \in\mathcal{S}(K)$. This yields
  \begin{align}\label{contra gb}
\mathbb{E}\left(e^{iu\int_K p_t(y,x) Y(x) \mu(dx)} \right)=\textup{e}^{- \left|u\right|^\alpha \|(-\Delta)^{-s}(p_t (y,\cdot)-1)\|_{L^\alpha(K,\mu)}^\alpha}
  \end{align}
  If $s\le \max\left(\frac{(\alpha-1)d_h}{\alpha d_w},0\right)$, then $\|(-\Delta)^{-s}(p_t (y,\cdot)-1)\|_{L^\alpha(K,\mu)} \to +\infty$ when $t \to 0$. Indeed, applying Lemma~\ref{riesz-kernel laplace}, since $s>0$, we first note that
  \[
  (-\Delta)^{-s}(p_t (y,\cdot)-1)(x)=\int_K G_s(x,z)(p_t(y,z)-1) \mu(dz)=\int_K G_s(x,z)p_t(y,z) \mu(dz). 
  \]
  Then, we have the estimates for $x,y\in K$, $x\ne y$,
$$
 G_s(x,y) \ge C_1 d(x,y)^{sd_w-d_h} -C_2,
$$
by Proposition \ref{Control-G}, 
and by \eqref{eq:subGauss-upper},
$$
p_t(y,z) \ge c_{1}t^{-d_{h}/d_{w}}\exp\biggl(-c_{2}\Bigl(\frac{d(y,z)^{d_{w}}}{t}\Bigr)^{\frac{1}{d_{w}-1}}\biggr) 
$$ 
which imply 
\begin{align*}
    \int_K G_s(x,z)p_t(y,z) \mu(dz) & \ge c_3 t^{-d_{h}/d_{w}} \int_K \exp\biggl(-c_{2}\Bigl(\frac{d(y,z)^{d_{w}}}{t}\Bigr)^{\frac{1}{d_{w}-1}}\biggr) d(x,z)^{sd_w-d_h} \mu (dz)-c_4 \\
     &\ge c_5 t^{-d_{h}/d_{w}} \int_{B(y,t^{1/d_w})} d(x,z)^{sd_w -d_h}\mu(dz)-c_4 \\
     &\ge  c_6  (d(x,y)+t^{1/d_w})^{sd_w -d_h}-c_4,
\end{align*}
where the last inequality follows from the Ahlfors property \eqref{Ahlfors} and the triangle inequality since $sd_w-d_h<0$. 
We then see that for $s\le \max\left(\frac{(\alpha-1)d_h}{\alpha d_w},0\right)$
\[
\lim_{t \to 0}\int_K (d(x,y)+t^{1/d_w})^{\alpha(sd_w -d_h)} \mu(dx)=+\infty.
\]
Therefore,  if $s\le \max\left(\frac{(\alpha-1)d_h}{\alpha d_w},0\right)$, then $\|(-\Delta)^{-s}(p_t (y,\cdot)-1)\|_{L^\alpha(K,\mu)}\to +\infty$ when $t \to 0$. From \eqref{contra gb} we deduce that for $s\le \max\left(\frac{(\alpha-1)d_h}{\alpha d_w},0\right)$ we have for every $y \in K$ and $u\in\mathbb{R}\backslash\{0\},$
 \[
\lim_{t \to 0} \mathbb{E}\left(e^{iu\int_K p_t(y,x) Y(x) \mu(dx)} \right)=0.
 \]
 This is not possible because  a.s, when $t \to 0$,  $y \to \int_K p_t(y,x) Y(x) \mu(dx)$ converges in $L^1(K,\mu)$ to $Y$, see \cite[Page 114]{BauBook} or \cite[Theorem 1.4.1.]{Davies}, since we have a.s.
  \[
\int_K |Y(x)| \mu(dx) <+\infty .
  \]
  We conclude that  $s > \max\left(\frac{(\alpha-1)d_h}{\alpha d_w},0\right)$. In that case, from the first part of the proof we deduce then that
for every $ f\in \mathcal{S}(K)$, we have a.s.
  \[
\int_K f(x) Y(x) \mu(dx)=X_{s,\alpha}(f)=\int_K f(x) \widetilde{X}_{s,\alpha}(x) \mu(dx). 
  \]
  This implies that  we have a.s. for a.e. $x \in K$
  \[
  Y(x)=\widetilde{X}_{s,\alpha}(x) 
  \]
  since as before we can use $f=p_t(x,\cdot)-1$ and let $t \to 0.$
\end{proof}

When $\alpha <2$, the study of the sample path properties of the pointwise Neumann fractional $\alpha$-stable random field $
\widetilde{X}_{s,\alpha}$ is based on its representation as a conditionally Gaussian LePage series. It follows  from the series representation  \eqref{LePageSeries} of any stochastic integral $W_{K,\alpha}(f)$ and Equation \eqref{SeriesVA}. 
\begin{proposition} \label{FSM_SeriesRepresentation} 
Let $\alpha \in (0,2)$ and  $s>\max\left(\frac{(\alpha-1)d_h}{\alpha d_w},0\right)$. Assume that Assumption \ref{AssumptionLePage} is fulfilled and let us consider the Neumann fractional  $\alpha$-stable field $\widetilde{X}_{s,\alpha}$ defined by \eqref{dens def}. 
Then for any $x\in K$, the series 
\begin{equation}
\label{Field-Y-Series}
Y_{s,\alpha}(x)=S\left(G_s\left(x,\cdot\right)\right)=D_\alpha \sum_{n=1}^{+\infty} T_n^{-{1/\alpha}} G_s\left(x,\xi_n\right)g_n,
\end{equation}
with $D_\alpha$ given by \eqref{csteCalpha}, converges almost surely and
$$
\left(\widetilde{X}_{s,\alpha}(x)\right)_{x\in K}\stackrel{(fdd)}{=} \left(Y_{s,\alpha}(x)\right)_{x\in K}.
$$
\end{proposition}

Then conditionally to $(T_n,\xi_n)_{n\ge 1}$, $Y_{s,\alpha}$ is a centered real-valued Gaussian random field and the  increment   $Y_{s,\alpha}(x)-Y_{s,\alpha}(y)$ is a centered real-valued Gaussian variable with variance 
\begin{equation}
\label{var totale}
s_{\alpha}(x,y)^2=D_\alpha^2\,\mathbb{E}\left(g_1^2\right)\sum_{n=1}^{+\infty} T_n^{-2/\alpha} 
\left|G_s(x,\xi_n)-G_s(y,\xi_n)\right|^2. 
\end{equation}
\subsection{Sample paths regularity }

In this section, we are now interested in the  sample path smoothness of the Neumann fractional $\alpha$-stable  random field $X_{s,\alpha}$ on $K$.  This extend the sample path study done in the Gaussian framework in \cite{BL22-FGF}, considering the whole range of parameter $s$. The proof is not based on the control of the variance of the increments, since the stable field $X_{s,\alpha}$ is not square integrable but as in several studies about stable random fields  (see e.g. \cite{KM91,BL12,BLS11-MOSRF,BL-2015-Bernoulli}), one main ingredient is  the LePage series representation stated in Proposition \ref{FSM_SeriesRepresentation}.
\begin{theorem}
\label{SPR}
 Let $\alpha\in (0,2)$, $s>\max\left(\frac{(\alpha-1)d_h}{\alpha d_w},0\right)$ and let  $\widetilde{X}_{s,\alpha}$ be the Neumann fractional $\alpha$-stable field. Let us set $\eta_s=\min(s,1)d_w-d_h$ and 
$$
\beta_s=\left\{\begin{array}{ll}
0&\textrm{{if} $s< 1$}\\
1&\textrm{{if} $s\ge 1$.}
\end{array}
\right. $$

 \begin{enumerate}
\item If  $s>\frac{d_h}{d_w}$, 
 then there exists a modification $\widetilde{X}_{s,\alpha}^*$ of $\widetilde{X}_{s,\alpha}$ such that  
	\begin{equation}
	\label{ModContX}
	\lim_{\delta \to 0}\underset{\underset{x,y\in K}{\scriptstyle 0<d(x,y)}\le \delta}{\sup}\  \frac{\left| \widetilde{X}_{s,\alpha}^*(x)-\widetilde{X}_{s,\alpha}^*(y)\right|}{d(x,y)^{\eta_s} |\ln d(x,y)|^{\beta_s+\frac{1}{2}}} <\infty.	
	\end{equation} 
 \item If $s\le \frac{d_h}{d_w}$, a.s., the sample paths of $\widetilde{X}_{s,\alpha}$ are unbounded on the compact $K$. 
 \end{enumerate}
\end{theorem}

\begin{remark} \hfill
\begin{enumerate}
\item For $\alpha=2$, i.e. for the Neumann fractional Gaussian field $\widetilde{X}_{s,2}$, Equation \eqref{ModContX} has been stated in Theorem 3.8 of \cite{BL22-FGF} when $\frac{d_h}{2d_w}<s<1-\frac{d_h}{2d_w}$ and can be extended to any $s>\frac{d_h}{2d_w}$ following \cite{BC22-DFGF} which deals with the Dirichlet fractional Laplacian. 
\item Theorem 3.8 extends the known results for moving average  and log-fractional stable fields indexed by the Euclidean space $\mathbb R^d$. Actually,  an analogous of the upper bound \eqref{ModContX} has been established in \cite{KM91} when $d=1$,   $1<\alpha<2$ and $1/\alpha<H<1$ (which corresponds to $\frac{1}{2}<s< 1-\frac{1}{2\alpha}$). In all the other cases, moving average fractional stable random fields have almost surely unbounded sample paths on any non empty open subset of $\mathbb R^d$ and so have  log-fractional stable random fields, see \cite{taqqu,MR2531090}. 
\item As in the Euclidean framework, contrary to  Neumann fractional Gaussian random fields  ($\alpha=2$), when $\alpha<2$, the Neumann fractional stable random field $\widetilde{X}_{s,\alpha}$  does not always have continuous sample paths. Note also that  the restriction of any Euclidean moving average or log-fractional stable field is unbounded almost surely on the Sierpi\'nski gasket $K$ whereas the  Neumann fractional stable field $\widetilde{X}_{s,\alpha}$ has H\"olderian sample paths on $K$ as soon as $s>\frac{d_h}{d_w}$.  
\end{enumerate}  
\end{remark}

We divide the proof in several steps. We first assume that $s>d_h/d_w$ and  study   the sample paths regularity of  the series $Y_{s,\alpha}$,  which has the same finite distributions as $\widetilde{X}_{s,\alpha}$. To do so, we first consider  the partial sum $Y_{s,\alpha,N}$ defined by 
$$
Y_{s,\alpha,N}(x)=D_\alpha\sum_{n=1}^N T_n^{-1/\alpha} G_s(x,\xi_n) g_n, \quad x\in K. 
$$
Note that since $s>d_h/d_w$, by Theorem \ref{Holder-Riesz}, each $G_s(\cdot,\xi_n)$ is continuous on $K$ and then $Y_{s,\alpha,N}$ has continuous sample paths on $K$. Moreover, conditionally to $(T_n,\xi_n)_{n\ge 1}$, for any integer $N\ge 1$, $Y_{s,\alpha,N}$ is a centered real-valued Gaussian random field and the  increment   $Y_{s,\alpha,N}(x)-Y_{s,\alpha,N}(y)$ is a centered real-valued Gaussian variable with variance 
{
\begin{equation}
\label{var:Lepage:YN}
s_{\alpha,N}(x,y)^2=D_\alpha^2\,\mathbb{E}\left(g_1^2\right)\sum_{n=1}^{N} T_n^{-2/\alpha}  \left|G_s(x,\xi_n)-G_s(y,\xi_n)\right|^2 \le s_{\alpha}(x,y)^2,
\end{equation}
}
   where we recall that ${s_{\alpha}(x,y)^2}$ is given by \eqref{var totale}.

\begin{proposition} 
\label{ControlSums} Let $\alpha \in (0,2)$ and $s>\frac{d_h}{d_w}$. Let  $\eta_s$ and  $\beta_s$ be defined as  in Theorem \ref{SPR}. 

\begin{enumerate}
\item There exist  $h_s>0$ and  a finite positive random variable $C$ such that  a.s for any $x,y\in K$ such that $d(x,y)\le h_s$, 
\begin{equation}
\label{ControlPartialSums2}
\sup_{N\ge 1} \left| Y_{s,\alpha,N}(x)-Y_{s,\alpha,N}(y)\right| \le C d(x,y)^{\eta_s}(-\ln d(x,y))^{\beta_s+\frac{1}{2}}. 
\end{equation}
\item Moreover, almost surely, the series $(Y_{s,\alpha,N})_{N\ge 1}$ converges uniformly on $K$ to $Y_{s,\alpha}$, and  $Y_s$ also satisfies~\eqref{ControlPartialSums2}, that is   a.s for any $x,y\in K$ such that $d(x,y)\le h_s$, 
\begin{equation}
\label{ControlY}
{\left| Y_{s,\alpha}(x)-Y_{s,\alpha}(y)\right| \le C d(x,y)^{\eta_s}(-\ln d(x,y))^{\beta_s+\frac{1}{2}}.}
\end{equation}
\end{enumerate}
\end{proposition}

The proof relies on two main ingredients. The following lemma which gives an upper bound  for the conditional variance $s_{\alpha,N}(x,y)^2$, uniformly in $N$,  and the Garsia-Rodemich-Rumsey inequality. 

\begin{lemma} 
\label{ControlVarCond} Let $\alpha \in (0,2)$ and $s>\frac{d_h}{d_w}$. Then there exists a finite positive random variable $A$ such that  a.s for any $x,y\in K$, 
$$\sup_{N\ge 1}
s_{\alpha,N}(x,y)  =s_{\alpha}(x,y)\le  A d(x,y)^{\eta_s} \max\left(\left|\ln  d(x,y)\right|^{\beta_s},1\right) 
$$
with $\eta_s$ and  $\beta_s$ given in Theorem \ref{SPR}. 
\end{lemma} 
\begin{proof} As a direct consequence of Theorem \ref{Holder-Riesz}, by definition of $\eta_s$ and $\beta_s$, 
$$
\sup_{N\ge 1}
s_{\alpha,N}(x,y)^2  =s_{\alpha}(x,y)^2\le A^2 d(x,y)^{2\eta_s} \max\left(\left| \ln d(x,y)\right|^{\beta_s},1\right)^2
$$
where $A^2=c\sum_{n=1}^{+\infty} T_n^{-2/\alpha} $, with $c$ a positive constant, converges almost surely since $2/\alpha>1$ (see e.g. Theorem 1.4.5 in \cite{taqqu}). 
\end{proof}

\vskip 5pt

Let us now recall the Garsia-Rodemich-Rumsey inequality for fractals, established by Barlow and Perkins in   Lemma 6.1 of \cite{BP88}. 
\begin{lemma}\label{BPlem} Let 
	$p$ be an increasing continuous function on $[0,\infty)$ such that $p(0)=0,$ and $\psi : \mathbb{R}\rightarrow  \mathbb{R}_+$ be a non-negative symmetric continuous convex function such that  $\lim_{u\to \infty} \psi(u)=\infty$. Let $f: K\rightarrow  \mathbb{R}$ be a measurable function such that 
	$$
	\Gamma =\int_{K\times K} \psi \left(\frac{\left|f(x)-f(y)\right|}{p(d(x,y))}\right)\mu(dx)\mu( dy)
	<\infty.$$
	Then there exists a constant $c_K$ depending only on  $d_h$ such that 
\begin{equation}
\label{GRRinequality}
	\left|f(x)-f(y)\right|\le 8\int_0^{d(x,y)} \psi^{-1}\left(\frac{c_K \Gamma}{u^{2d_h}}\right) \,p({\rm d}u)
	\end{equation}
	for $\mu\times \mu$-almost all $(x,y)\in K\times K$.  Moreover if $f$ is continuous on $K$, then \eqref{GRRinequality} holds for any $(x,y)\in K\times K$. 
	\end{lemma}

In the sequel, we choose  
\begin{equation}
    \label{choix psi}
\psi(u)=\exp\left(\frac{u^2}{4}\right)-1.
\end{equation} 
The choice of $p=p_s(d(x,y))$ is a slight modification of the upper bound stated in Lemma \ref{ControlVarCond} in order to obtain an increasing function.  
Let us first note  that since $\eta_s>0$, there exists $h_s\in(0,1/\textup{e}]$ small enough such that  the function $p_s$ defined by 
$$
p_s(h)=\left\{\begin{array}{ll}
0&\textrm{if $h=0$}\\
h^{\eta_s}\max\left(\left| \ln h\right|^{\beta_s},1\right)=h^{\eta_s}\left(- \ln h\right)^{\beta_s}&\textrm{if $h\in (0,h_s]$}\\
\end{array}
\right.
$$
is a continuous and increasing on $[0,h_s]$. {We then extend the definition of $p_s$  so that $p_s$ is still a continuous and increasing function  on the  interval $[0,1]$.} 
As a consequence, there exist some positive constants $d_1,d_2$ such that  for any $h\in (0,1],$
\begin{equation}
\label{comp-p}
d_1p_s(h)\le h^{\eta_s}\max\left(\left| \ln h\right|^{\beta_s},1\right)\le d_2 p_s(h).
\end{equation}

Then we will apply Garsia-Rodemich-Rumsey inequality choosing  $f=Y_{s,N}/(A d_2) $ with $A$ the positive random variable introduced in Lemma \ref{ControlVarCond}. This leads to study 
$$
\Gamma_N=\int_{K\times K}  \psi \left(\frac{\left|Y_{s,\alpha,N}(x)-Y_{s,\alpha,N}(y)\right|}{A d_2  \, p_s(d(x,y))}\right)\mu(dx)\mu(dy).
$$
\begin{lemma}  \label{ControlGammaN} 
Let $\alpha \in (0,2)$ and $s>\frac{d_h}{d_w}$. Then, with the previous notations, 
$
\mathbb{E}\left(\sup_{N\ge 1} \Gamma_N\right)<+\infty. 
$
\end{lemma}

\begin{proof}   
For any integer $M\ge 1$, 
$$
\sup_{1\le N\le M}
\Gamma_N \le  \int_{K\times K}  \psi \left(\sup_{1\le N\le M}\frac{\left|Y_{s,\alpha,N}(x)-Y_{s,\alpha,N}(y)\right|}{{A}d_2  p_s(d(x,y))}\right)\mu(dx)\mu(dy)
$$
 since  the function $\psi$ given by \eqref{choix psi} {is} a non-decreasing function.  Then, by Fubini-Tonelli theorem, 
$$
\mathbb{E}\left(\sup_{1\le N\le M} \Gamma_N\right)\le \int_{K\times K} I_M(x,y)\mu(dx)\mu(dy)
 $$
 where for $x\ne y$,
 $$
 I_M(x,y)=\mathbb{E}\left( \psi \left(\sup_{1\le N\le M}\frac{\left|Y_{s,\alpha,N}(x)-Y_{s,\alpha,N}(y)\right|}{{A}d_2 \, p_s(d(x,y))}\right)\right). 
 $$
Moreover, since $\psi$ is non-negative,
$$
I_M(x,y)=\displaystyle \int_{0}^{+\infty} \mathbb{P}\left( \psi \left(\sup_{1\le N\le M}\frac{\left|Y_{s,\alpha,N}(x)-Y_{s,\alpha,N}(y)\right|}{{A}d_2  p_s(d(x,y))}\right)> t\right) dt
$$
so that, by definition \eqref{choix psi} of $\psi$ and a simple change of variables, 
$$
I_M(x,y)=\int_1^{+\infty} \mathbb{P}\left( \sup_{1\le N\le M}\left|Y_{s,\alpha,N}(x)-Y_{s,\alpha,N}(y)\right|> 2 {A}d_2  p_s(d(x,y)) \sqrt{\ln t}
\right) dt.
$$
Now by Lemma \ref{ControlVarCond} and Equation \eqref{comp-p}, since $d(x,y)\le 1$,
$$
{A}d_2  p_s(d(x,y))\ge {s_{\alpha,M}(x,y)}
$$
{where we recall that $s_{\alpha,M}(x,y)$ is given by \eqref{var:Lepage:YN}}
and then  
$$
I_M(x,y)=\int_1^{+\infty} \mathbb{P}\left( \sup_{1\le N\le M}\left|Y_{s,\alpha,N}(x)-Y_{s,\alpha,N}(y)\right| > 2{s_{\alpha,M}(x,y)} \sqrt{\ln t}
\right) dt.
$$
Then,   conditioning to $(T_n,\xi_n)_{n\ge 1}$ and  applying {Proposition 2.3 of \cite{Ledoux-Talagrand}} lead to 
$$
I_M(x,y)
 \le \int_1^{+\infty}\mathbb{E}\left(  \mathbb{P}\left( \left. \left|Y_{s,\alpha,M}(x)-Y_{s,\alpha,M}(y)\right| > 2{s_{\alpha,M}(x,y)} \sqrt{\ln t}\, \right| (T_n,\xi_n)_{n\ge 1}
\right)\right) dt
$$
since the random variables $g_n$ are i.i.d. and symmetric. Therefore,  
since, conditionally to $(T_n,\xi_n)_{n\ge 1}$,  $Y_{s,\alpha,M}(x)-Y_{s,\alpha,M}(y)$ is a centered real-valued Gaussian variable with variance ${s_{\alpha,M}(x,y)^2}$, 
$$
I_M(x,y)
 \le \int_1^{+\infty}  \mathbb{P}\left( |Z| > 2 \sqrt{\ln t}\right) dt
$$
with $Z$ a standard normal random variable. Moreover, 
$$\forall \lambda >0, \  
\mathbb{P}\left( |Z| > \lambda\right)\le  \sqrt{\frac{2}{\pi}}\frac{\textup{e}^{-\frac{\lambda^2}{2}}}{\lambda}
$$
and so 
$$
I_M(x,y)\le \frac{1}{\sqrt{2\pi}}\int_{ 1}^{+\infty} \frac{1}{t^2 \sqrt{\ln t}} \, dt :=c <+\infty.
$$
Therefore,  
$$
\mathbb{E}\left(\sup_{1\le N\le M} \Gamma_N\right)\le c \mu(K)^2=c
$$
and then   by the monotone convergence theorem, 
$
\mathbb{E}\left(\sup_{ N\ge 1} \Gamma_N\right)\le c  <+\infty.
$ 
\end{proof} 

\medskip

We are now ready to prove Proposition \ref{ControlSums}. 

\begin{proof}[Proof of Proposition \ref{ControlSums}]\hfill
\begin{enumerate}
\item Let $\tilde{\Omega}=\left\{\sup_{N\ge 1} \Gamma_N <+\infty\right\}$.  By Lemma \ref{ControlGammaN}, $\mathbb{P}\left(\tilde{\Omega}\right)=1$ and then it is sufficient to establish~\eqref{ControlPartialSums2} for $\omega\in\tilde{\Omega}$.  So let from now $\omega \in \tilde{\Omega}$. 
Then by Garsia-Rodemich-Rumsey inequality (see Lemma~\ref{BPlem}), for each integer $N\ge 1$,   
$$
\forall x,y\in K, \ \left|Y_{s,\alpha,N}(x,\omega)-Y_{s,\alpha,N}(y,\omega)\right|\le 8A d_2\int_0^{d(x,y)} \psi^{-1}\left( \frac{c_K\Gamma_N(\omega)}{u^{2d_h}}\right) p_s(du)
$$
 since each $f_N=Y_{s,\alpha,N}/(A d_2)$ is continuous on $K$.  
Then since $\psi$ is defined by Equation \eqref{choix psi}, 
$$
\forall x,y\in K, \ \sup_{N\ge 1} \left|Y_{s,\alpha,N}(x,\omega)-Y_{s,\alpha,N}(y,\omega)\right|\le 16 A d_2\int_0^{d(x,y)} \sqrt{\ln\left(1+\frac{c_K \Gamma_\infty(\omega)}{u^{2d_h}}\right)}p_s(du)
$$
where $\Gamma_\infty(\omega)=\sup_{N\ge 1} \Gamma_N (\omega)<+\infty$.  
Note that  for any ${u \in (0,1/\textup{e}]}$
$$
0\le16 A d_2\sqrt{ \ln\left(1+\frac{c_K \Gamma_\infty} {u^{2d_h}}\right)}\le C \sqrt{|\ln u|}
$$
with $C$ a finite positive random variable which does not depend on $u$.  Hence, for any $x,y$ such that $d(x,y)\le 1/\textup{e}$,
$$
  \sup_{N\ge 1} \left|Y_{s,\alpha,N}(x,\omega)-Y_{s,\alpha,N}(y,\omega)\right|\le C(\omega)\int_0^{d(x,y)}\sqrt{|\ln u|}\, p_s(du). 
$$
Moreover, denoting by $c$ a  constant which value may change at each line   and applying two integrations by parts,  for any $r\in (0,h_s]\subset (0,1/\textup{e}]$, we have: 
$$
\begin{array}{rcl}
\displaystyle \int_0^{r}\sqrt{|\ln u|}\, p_s(du)& \le & \displaystyle c r^{\eta_s}\left(-\ln r\right)^{\beta_s+\frac{1}{2}}+c r^{\eta_s}\left(-\ln r\right)^{\beta_s-\frac{1}{2}}+c \int_{0}^r u^{\eta_s-1}\left(-\ln u\right)^{\beta_s-\frac{3}{2}} du \\[10pt]
&\le &\displaystyle  c r^{\eta_s}\left(-\ln r\right)^{\beta_s+\frac{1}{2}}+ c \int_{0}^r u^{\eta_s-1} du \\[10pt]
&\le & \displaystyle c r^{\eta_s}\left(-\ln r\right)^{\beta_s+\frac{1}{2}}+c r^{\eta_s} 
\\[10pt]
&\le &\displaystyle c r^{\eta_s}\left(-\ln r\right)^{\beta_s+\frac{1}{2}}
\end{array}
$$
since $\eta_s>0$ and $\beta_s-3/2<0$. 
Hence for $x,y\in K$ such that $d(x,y)\le h_s$,  up to change the value of~$C$, 
$$
\displaystyle  \sup_{N\ge 1} \left|Y_{s,\alpha,N}(x,\omega)-Y_{s,\alpha,N}(y,\omega)\right|\le \displaystyle  C(\omega) d(x,y)^{\eta_s}(-\ln d(x,y))^{\beta_s+\frac{1}{2}}, 
$$
which concludes the proof of Assertion 1.
\item   Since $K$ is a separable set, there exists a dense countable set $\mathcal{D}\subset K$.  We now follow the fourth step of the proof of Theorem 3.1 in \cite{BL-2015-Bernoulli}. Let 
$$
\Omega'=\tilde{\Omega}\cap\bigcap_{x\in \mathcal{D}} \left\{\lim_{N\to +\infty} Y_{s,\alpha,N}=Y_{s,\alpha} \right\}.
$$
By Proposition \ref{FSM_SeriesRepresentation} and  since $\mathcal{D}$ is countable,  
$\mathbb{P}\left(\Omega'\right)=\mathbb{P}\left(\tilde{\Omega}\right)=1.
$
Moreover, for each $\omega \in \Omega'$, according to the proof of Assertion 1, the real-valued sequence $(Y_{s,\alpha,N}(\cdot, \omega))_{N\ge 1}$ satisfies \eqref{ControlPartialSums2} and then is uniformly equicontinuous on $K$ since $\eta_s>0$.  Then, by Theorem I.26 and by generalizing Theorem~I.27 of \cite{Reed} to the compact set $K$, for each  $\omega \in \Omega'$, $(Y_{s,\alpha,N}(\cdot, \omega))_{N\ge 1}$ converges uniformly on~$K$ toward its limits $Y_{s,\alpha}(\cdot, \omega)$. Moreover, Equation \eqref{ControlPartialSums2} is fulfilled for any $\omega\in \Omega'$, which leads to Equation \eqref{ControlY} letting $N\to +\infty$  and concludes the proof. 
\end{enumerate}
\end{proof}
Let us now prove Theorem \ref{SPR}. 

\begin{proof}[Proof of Theorem \ref{SPR}]\hfill
\begin{enumerate}
\item Let us assume that $s>d_h/d_w$. We follow the same lines as the end of the proof of Proposition 5.1 in \cite{BL-2015-Bernoulli}. Let us consider  a dense countable set $\mathcal{D}\subset K$ and 
$ \Omega'':={\left\{ \widetilde{C}<\infty \right\}}$ with$$
\widetilde{C}:=\sup_{\underset{0<d(x,y)\le h_s}{x,y\in D}} \frac{\left| \widetilde{X}_{s,\alpha}(x)- \widetilde{X}_{s,\alpha}(y)\right|}{ d(x,y)^{\eta_s}\left(-\ln d(x,y)\right)^{\beta_s+\frac{1}{2}}} . 
$$
Then, since $ \widetilde{X}_{s,\alpha}$ and $Y_{s,\alpha}$ have the same finite distributions, and since $Y_{s,\alpha}$ satisfies \eqref{ControlY},  
$
\mathbb{P}\left(  \Omega''\right)=1
$ 
and $ \widetilde{X}_{s,\alpha}$ is stochastically continuous. 

Let us now define a modification $\widetilde{X}_{s,\alpha}^*$ of $\widetilde{X}_{s,\alpha}$ with Hölder sample paths.  First, for $\omega \notin \Omega''$, let $\widetilde{X}_{s,\alpha}^*(u,\omega)=0$ for any $u\in K$. Then let us consider $\omega \in \Omega''$. We set 
$$
\forall u\in\mathcal{D}, \ \widetilde{X}_{s,\alpha}^*(u,\omega)=\widetilde{X}_{s,\alpha}(u,\omega)  .
$$
Then, by definition of $\widetilde{C}$, for any $x,y\in \mathcal{D}$ such that $d(x,y)\le h_s$, 
\begin{equation}
\label{controlD}
\left| \widetilde{X}^*_{s,\alpha}(x,\omega)-\widetilde{X}^*_{s,\alpha}(y,\omega)\right| \le \widetilde{C}(\omega) d(x,y)^{\eta_s}\left(-\ln d(x,y)\right)^{\beta_s+\frac{1}{2}}. 
\end{equation}

Let us now consider $u\in K$ and $(u_n)_{n\in \mathbb N}$ be a sequence of $\mathcal D$ such that $\lim_{n\to +\infty} u_n=u$. 

Since $\widetilde{C}(\omega)<+\infty,$ Equation \eqref{controlD} implies that
the sequence $\left(\widetilde{X}_{s,\alpha}^*(u_n,\omega)\right)_{n\in \mathbb N}$ is a  real-valued Cauchy sequence and then converges.  Note also that the limits does not depend on  choice of $(u_n)_n$, so  that we can define $ \widetilde{X}^*_{s,\alpha}(u,\omega)$ by 
$$
\widetilde{X}^*_{s,\alpha}(u,\omega)=\lim_{n\to +\infty} \widetilde{X}_{s,\alpha}\left(u_n,\omega\right). 
$$
Then,  by Equation \eqref{controlD}, by definition of $\widetilde{X}_{s,\alpha}^*$ and continuity of the distance $d$, 
$$
\left| \widetilde{X}^*_{s,\alpha}(x,\omega)-\widetilde{X}^*_{s,\alpha}(y,\omega)\right|\le \widetilde{C}(\omega) d(x,y)^{\eta_s}\left(-\ln d(x,y)\right)^{\beta_s+\frac{1}{2}} 
$$
for any $x,y\in K$ such that $d(x,y)\le h_s$.  This inequality also holds for $\omega \notin \Omega ''$ up to replace $\widetilde{C}(\omega)$ by $0$, so that $\widetilde{X}_{s,\alpha}^*$ satisfies Equation \eqref{ModContX}. Moreover, by stochastic continuity of $\widetilde{X}_{s,\alpha}$ and since $\mathbb{P}(\Omega'')=1$,  the random field $\widetilde{X}_{s,\alpha}^*$ is a modification of $\widetilde{X}_{s,\alpha}$, which concludes the proof of Assertion 1.

\item Let us now assume that $s\le d_h/d_w$ and consider $\mathcal{D}\subset K$ a countable dense subset of $K$. Then, by the lower bounds in Lemma \ref{Control-G}, for any $y\in K\backslash\mathcal{D}$
$$
\sup_{x\in\mathcal{D}} \left|G_s(x,y)\right|=+\infty
$$
so that, since $\mu\left(K\backslash\mathcal{D}\right)>0$, 
$$
\int_{K}\sup_{x\in\mathcal{D}} \left|G_s(x,y)\right|^\alpha \mu(dy)=+\infty.
$$
Therefore, according to Theorem 10.2.3 in \cite{taqqu}, a.s. the sample paths of $\widetilde{X}_{s,\alpha}$ are unbounded on $K$.  
\end{enumerate}
\end{proof}

Finally, when $\alpha <1$, the upper bound \eqref{ModContX} of the modulus of continuity stated in Theorem \ref{SPR} for $s>\frac{d_h}{d_w}$ can be improved:   the $\frac{1}{2}$ in the logarithm term can be removed. Actually, since $\alpha<1$, the series \eqref{Field-Y-Series} converges absolutely almost surely and then the increments $Y_{s,\alpha}(x)-Y_{s,\alpha}(y)$ can be directly dominated, without controlling the conditional  variance $s_\alpha(x,y)^2$. It is then unnecessary to apply the Garsia-Rodemich-Rumsey inequality, which  gives the $(-\ln d(x,y))^{\frac{1}{2}}$ term in the upper bound \eqref{ModContX}. 
\begin{theorem}
    \label{SPR alpha >1}
 Let $\alpha\in (0,1)$  and $s>\frac{d_h}{d_w}$. Let  $\widetilde{X}_{s,\alpha}$ be the Neumann fractional $\alpha$-stable field, and  $\eta_s$ and $\beta_s$ be defined as in Theorem \ref{SPR}.  Then there exists  a  the modification $\widetilde{X}^*_{s,\alpha}$ of  $\widetilde{X}_{s,\alpha}$ such that 
  $$
\lim_{\delta \to 0}\underset{\underset{x,y\in K}{\scriptstyle 0<d(x,y)}\le \delta}{\sup}\  \frac{\left| \widetilde{X}_{s,\alpha}^*(x)-\widetilde{X}_{s,\alpha}^*(y)\right|}{d(x,y)^{\eta_s} |\ln d(x,y)|^{\beta_s }}<\infty.
	$$
\end{theorem}
\begin{proof} For any $x,y\in K$, 
$$
\left|Y_{s,\alpha}(x)-Y_{s,\alpha}(y)\right|\le D_\alpha\sum_{n=1}^{+\infty}T_n^{-1/\alpha} \left| G_s(x,\xi_n)-G_s(y,\xi_n)\right|\left|g_n \right| 
$$
and then by Theorem \ref{Holder-Riesz}, there exists a constant $c>0$ such that for any $x,y\in K$
$$
\left|Y_{s,\alpha}(x)-Y_{s,\alpha}(y)\right| \le c  d(x,y)^{\eta_s} \max\left(\left| \ln d(x,y)\right|^{\beta_s},1\right) \sum_{n=1}^{+\infty}T_n^{-1/\alpha} \left|g_n \right|.
$$
In addition, since $0<\alpha< 1$ and since Assumption \ref{AssumptionLePage} is fulfilled, by Theorem 1.4.5 in \cite{taqqu}, the series 
$$
A=c\sum_{n=1}^{+\infty}T_n^{-1/\alpha} \left|g_n \right|
$$
converges almost surely to an $\alpha$-stable random variable, and in particular $A<+\infty$ almost surely. Then the random field $Y_{s,\alpha}$ satisfies \eqref{ControlY} with $C=A$.  Hence Theorem \ref{SPR alpha >1} follows from the proof of Theorem \ref{SPR} (replacing $\beta_s+\frac{1}{2}$ by $\beta_s$). 
\end{proof}

\subsection{Symmetries and self-similarity properties}\label{invariance fields}

\subsubsection{Invariance by reflections}
Let  $\sigma_0,\sigma_1,\sigma_2$ be the reflections about the lines dividing the triangle with vertices $q_0,q_1,q_2$ into two equal parts. The Sierpi\'nski gasket is invariant by these reflections, and so is its Hausdorff measure $\mu$.  This leads to invariance of the random measure $W_{K,\alpha}$ and then of the fractional stable random fields ${X}_{s,\alpha}$ and $\widetilde{X}_{s,\alpha}$. Such properties have already been established in \cite{BL22-FGF}, studying the covariance function when $\alpha=2$. 

\begin{proposition} Let $\alpha\in (0,2]$ and $s>0$. \hfill

\begin{enumerate}
\item For any $f\in L^\alpha(K,\mu)$,  for every $i\in\{0,1,2\}$,  $f\circ \sigma_i\in L^\alpha(K,\mu)$ and 
$$
W_{K,\alpha}(f\circ \sigma_i)\stackrel{(d)}{=}W_{K,\alpha}(f).
$$
\item As a consequence,  for every $i\in\{0,1,2\}$, and every $s>0$
$$
\left({X}_{s,\alpha}\left(f\circ\sigma_i(x)\right)\right)_{f\in H^{-s}(K)}\stackrel{(fdd)}{=}\left({X}_{s,\alpha}\left(f\right)\right)_{f\in H^{-s}(K)}. 
$$
\item Moreover, if $s>\max\left(\frac{(\alpha-1)d_h}{\alpha d_w},0\right)$, then for every $i\in\{0,1,2\}$, 
$$
\left(\widetilde{X}_{s,\alpha}\left(\sigma_i(x)\right)\right)_{x\in K}\stackrel{(fdd)}{=}\left(\widetilde{X}_{s,\alpha}\left(x\right)\right)_{x\in K}. 
$$
\end{enumerate}
\end{proposition}

\begin{proof} \hfill

\begin{enumerate}
\item Let  $f\in L^\alpha(K,\mu)$ and $i\in \left\{0,1,2\right\}$. Note that  by the change of variables $y=\sigma_i(x)$, 
$$\int_K \left|f\circ\sigma_i(x)\right|^\alpha \, \mu(dx)=\int_K \left|f(y)\right|^\alpha \, \mu(dy) <+\infty 
$$
since $K$ and $\mu$ are invariant by $\sigma_i$. In particular, $f\circ\sigma_i \in L^\alpha(K,\mu)$. Moreover,   
$$
\left\|W_{K,\alpha}(f\circ \sigma_i)\right\|_{L^\alpha(K,\mu)}^\alpha=\int_K \left|f\circ\sigma_i(x)\right|^\alpha \, \mu(dx)=\int_K \left|f(y)\right|^\alpha \, \mu(dy)=\left\|W_{K,\alpha}(f)\right\|_{L^\alpha(K,\mu)}^\alpha. 
$$
Hence, the two   $S\alpha S$ random variables  $W_{K,\alpha}(f\circ \sigma_i)$ and $W_{K,\alpha}(f)$ have the same distribution, which concludes the proof of the first assertion. 
\item Let $i\in\{0,1,2\}.$  Since $\Phi_j\in\mathcal{F},$ $j\ge 1$ is an orthonormal basis of $L_0^2(K,\mu)$, using the invariance of $\mu$ by the continuous bijective function $\sigma_i$, one can check that $\Phi_j\circ \sigma_i \in\mathcal{F},$ $j\ge 1$, is also an orthonormal basis of $L_0^2(K,\mu)$.  Moreover, since the Dirichlet form $\mathcal{E}$ is invariant by $\sigma_i$, that is since 
$$
\mathcal{E}\left(f\circ \sigma_i,f\circ \sigma_i\right)=\mathcal{E}\left(f,f\right),
$$
we have: 
$$
\Delta \left(f\circ \sigma_i\right)= \left(\Delta f\right)\circ \sigma_i.
$$
and then 
$$
\Delta \left(\Phi_j\circ \sigma_i\right)= \left(\Delta \Phi_j\right)\circ \sigma_i=-\lambda_j \Phi_j\circ \sigma_i
$$
that is, as the function $\Phi_j$,  $\Phi_j\circ \sigma_i$ is an eigenfunction of $-\Delta$ associated to the eigenvalue $-\lambda_j$. By unicity of the fractional Laplacian $(-\Delta)^{-s}$, for any $g\in L_0^2(K,\mu)$, in the basis  $\Phi_j\circ \sigma_i \in\mathcal{F},$ $j\ge 1$, 
$$
(-\Delta)^{-s} g=\sum_{j=1}^{+\infty} \frac{1}{\lambda_j^s} \left(\int_K \Phi_j(\sigma_i(y)) g(y) \mu (dy)\right) \Phi_j\circ \sigma_i. 
$$ 

Hence, since for any $f\in L_0^2(K,\mu)$, $g=f\circ \sigma_i\in L_0^2(K,\mu)$ and we have:
$$
(-\Delta)^{-s} (f\circ \sigma_i)=\sum_{j=1}^{+\infty} \frac{1}{\lambda_j^s} \left(\int_K \Phi_j(\sigma_i(y)) \, (f\circ\sigma_i)(y) \mu (dy)\right)\Phi_j\circ \sigma_i.
$$
Then, applying the change of variable $x=\sigma_i(y)$ in the integrals and since $\mu$ is invariant by $\sigma_i$, we get: for any $f \in L_0^2(K,\mu)$,
\begin{equation}
\label{Laplacian-symmetry}
(-\Delta)^{-s} (f\circ \sigma_i)=\sum_{j=1}^{+\infty} \frac{1}{\lambda_j^s} \left(\int_K \Phi_j(x) f(y) \mu (dy)\right) \Phi_j\circ \sigma_i = \left((-\Delta)^{-s} f\right)\circ \sigma_i.
\end{equation}
Moreover from the definition of the Sobolev space $H^{-s}(K)$, it follows that for $f\in H^{-s}(K)$, $f\circ \sigma_i \in H^{-s}(K)$ and $\| f \|_{H^{-s}(K)}=\| f \circ \sigma_i \|_{H^{-s}(K)}$ and that Equation \eqref{Laplacian-symmetry} still holds (by density) for $f\in H^{-s}(K)$. Hence, for any $f\in H^{-s}(K)$, the two random variables $X_{s,\alpha}(f)$ and $ X_{s,\alpha}(f\circ \sigma_i)$ are well-defined   and we have: 
$$
X_{s,\alpha}\left(f \circ \sigma_i\right)=W_{K,\alpha}\left((-\Delta)^{-s}\left(f \circ \sigma_i\right)\right)=W_{K,\alpha}\left(\left((-\Delta)^{-s}f \right)\circ \sigma_i\right)
$$
and then, by Assertion 1., 
$$
X_{s,\alpha}\left(f \circ \sigma_i\right)\stackrel{(d)}{=} W_{K,\alpha}\left((-\Delta)^{-s}f \right),
$$
that is 
$$
X_{s,\alpha}\left(f \circ \sigma_i\right)\stackrel{(d)}{=} X_{s,\alpha}\left(f \right),
$$
which implies Assertion 2.  by linearity of $X_{s,\alpha}$ on $H^{-s}(K)$.

\item Let $s>\max\left(\frac{(\alpha-1)d_h}{\alpha d_w},0\right)$ and let $i\in \{0,1,2\}$. Then, for any $x\in K$, 
$$
\widetilde{X}_{s,\alpha}\left(\sigma_i(x)\right)=\int_{K} G_s\left(\sigma_i(x),y\right) W_{K,\alpha}(dy). 
$$
Hence, by linearity of the stochastic integrals with respect to  $W_{K,\alpha}$ and by Assertion 1, 
$$
\left(\widetilde{X}_{s,\alpha}\left(\sigma_i(x)\right)\right)_{x\in K}\stackrel{(fdd)}{=}  \left(  \int_{K} G_s\left(\sigma_i(x),\sigma_i(z)\right) W_{K,\alpha}(dz) \right)_{x\in K} .
$$
Moreover, according to the proof of Proposition 3.10 in \cite{BL22-FGF}, for any $x,z\in K$ such that $x\ne z$, 
\[
G_s\left(\sigma_i(x),\sigma_i(z)\right) =G_s(x,y).
\]
Therefore, 
$$
\left(\widetilde{X}_{s,\alpha}\left(\sigma_i(x)\right)\right)_{x\in K}\stackrel{(fdd)}{=}  \left(  \int_{K} G_s\left(x,z\right) W_{K,\alpha}(dz) \right)_{x\in K} ,
$$
which concludes the proof.
\end{enumerate}
\end{proof}
\subsubsection{Invariance by scaling}

Let $w=(i_1, \cdots, i_n) \in \{ 0,1,2 \}^n $, and denote $$F_w=F_{i_1} \circ \cdots \circ F_{i_n}$$ where we recall that 
\[
F_i(z)=\frac{1}{2}(z-q_i)+q_i.
\]
The compact set $K_w:=F_w(K) \subset K$ is itself a Sierpi\'nski gasket.  We equip $K_w$ with the restriction $\mu_{|_{K_w}}$ of the  Hausdorff measure $\mu$. Let us from now indicate with a superscript or subscript $w$ the objects related to the Sierpi\'nski
gasket $K_w$ (Dirichlet form $\mathcal{E}^w$, Laplacian $\Delta_w$, heat kernel $p^w_t(x,y)$, etc...). 

We now introduce the Neumann fractional $\alpha$-stable random field on the smaller Sierpi\'nski $K_w$ as previously done on~$K$. Let us first   note that the function $F_w:K\rightarrow K_w$ is continuous and bijective and that according to the proof of Proposition 3.10 in \cite{BL22-FGF},
$$
\mathcal{F}^w=\left\{f\circ F_w^{-1},\, f\in \mathcal{F}\right\}
$$
and for $f\in \mathcal{F}$, 
\begin{equation}
\label{Laplacian-self-similarity}
\Delta_w \left(f\circ F_w^{-1}\right)=5^n\left(\Delta f\right)\circ F_w^{-1} .
\end{equation}
Let us now consider the functions 
$$
\Phi_j^w=3^{\frac{n}{2}}\, \Phi_j\circ F_w^{-1} \in \mathcal{F}^w, \ j\ge 1. 
$$
Since  $\Phi_j$, $j\ge 1$ is an orthonormal basis of $L_0^2(K,\mu)$, from the self-similarity  of the Hausdorff measure~$\mu,$ it follows that $\Phi_j^w\in\mathcal{F}^w$, $j\ge 1$, is an orthonormal basis of $L_0^2\left(K_w,\mu\right)$.  Moreover, since  $\Delta \Phi_j=-\lambda_j \Phi_j$, by Equation \eqref{Laplacian-self-similarity}, 
  $$
  -\Delta_w \Phi_j^w=\lambda_j^w\Phi_j^w
  $$ with 
$\lambda_{j,w}=5^n \lambda_j$. Then,  the heat kernel admits the  spectral expansion: 
$$
p^w_t(x,y)=\frac{1}{\mu(K_w)}+\sum_{j=1}^{+\infty} \textup{e}^{-\lambda_{j,w}}\Phi_j^w(x)\Phi_j^w(y)=3^n+\sum_{j=1}^{+\infty} \textup{e}^{-\lambda_{j,w}}\,\Phi_j^w(x)\Phi_j^w(y). 
$$

We next  consider  the following space of test functions
\[
\mathcal{S}\left(K_w\right)= \left\{ f \in C_0\left(K_w\right),  \forall k \ge 0 \lim_{n \to +\infty} n^k \left| \int_{K_w}  \Phi_n^w(y) f(y) \mu(dy) \right| =0 \right\},
\] 
where
\[
C_0\left(K_w\right)=\left\{ f \in C\left(K_w\right), \int_{K_w} f d\mu=0 \right\}\subset L_0^2\left(K_w,\mu\right).
\]
We define then the Sobolev space $H^{-s}\left(K_w\right)$ as the completion of $\mathcal{S}\left(K_w\right)$ with respect to the norm $\| f \|_{H^{-s}(K_w)}=\|(-\Delta)^{-s} f \|_{L^2(K_w,\mu)}$. 
For $s\ge 0$,  we then consider $X_{s,\alpha}^w$  the fractional stable field defined on {$H^{-s}\left(K_w\right)$} by 
$$
X_{s,\alpha}^w(f)=\int_{K_w} \left(-\Delta_w\right)^{-s} f(z) \, W_{K_w,\alpha}(dz), \quad f\in H^{-s}\left(K_w\right)
$$
where $W_{K_w,\alpha}$ is a real-valued symmetric $\alpha$-stable random measure with control measure  $\mu_{|_{K_w}}$.   Note that the volume of $K_w$ is $\mu(K_w)=3^{-n}$, which  leads to  define the Riesz kernel by 
$$
G_s^w(x,y)=\int_{0}^{+\infty} t^{s-1} \left(p_{t}^w(x,y)-\frac{1}{\mu(K_w))}\right) dt=\int_{0}^{+\infty} t^{s-1} \left(p_{t}^w(x,y)-3^n\right) dt, \quad x,y\in K_w, \, x\ne y.
$$
Let us mention that there is a typo in the proof of Proposition 3.10 in\cite{BL22-FGF}, since the constant $3^n$ is missing in the definition of $G_s^w$. 

Finally, for $s>\max\left(\frac{(\alpha-1)d_h}{\alpha d_w},0\right)$, we also consider 
$$
\widetilde{X}_{s,\alpha}^w(x)=\int_{K_w} G_s^w(x,y) \, W_{K_w,\alpha}(dy), \quad x\in K_w
$$
and Theorem \ref{Existence-Density} still holds: that is $\widetilde{X}_{s,\alpha}^w$ is the density of $X_{s,\alpha}^w$ with respect to the measure $\mu_{|_{K_w}}$. Next Theorem, which generalizes the self-similarity property established in the Gaussian framework in \cite{BL22-FGF}, states that up to a normalization and a scaling, we recover the Neumann fractional $\alpha$-stable random field defined on the standard Sierpi\'nski gasket $K$.

\begin{proposition} Let $\alpha\in (0,2]$, $s>0$ and $H=sd_w-\frac{(\alpha-1)d_h}{\alpha}.$ 

\begin{enumerate}
\item For any $f\in L^\alpha\left(K_w,\mu_{|_{K_w}}\right)$,  then  $f\circ F_w\in L^\alpha(K,\mu)$ and 
$$
3^{n/\alpha} \,W_{K_w,\alpha}(f)\stackrel{(d)}{=}W_{K,\alpha}(f\circ F_w).
$$
 
\item As a consequence, 
$ \displaystyle
2^{nH}3^n\left({X}^w_{s,\alpha}\left(f\right)\right)_{f\in H^{-s}\left(K_w\right)}\stackrel{(fdd)}{=}\left({X}_{s,\alpha}\left(f\circ  F_w\right)\right)_{f\in H^{-s}\left(K_w\right)}. 
$ 

\item Moreover, for $s>\max\left(\frac{(\alpha-1)d_h}{\alpha d_w},0\right)$,  $(2^{nH} \widetilde{X}_{s,\alpha}^w( F_w (x)))_{x \in K}$ is the Neumann  fractional $\alpha$-stable  field   defined on $K$ by \eqref{dens def}, that is 
$$
2^{nH}(\widetilde{X}_{s,\alpha}^w( F_w (x)))_{x \in K}\stackrel{(fdd)}{=} (\widetilde{X}_{s,\alpha}( x))_{x \in K}.
$$

\end{enumerate}
\end{proposition}

\begin{remark} By analogy with the Fractional Brownian fields of index $H$, in view of the last proposition,  the random fields defined in Definition \ref{def fractional field} as a distribution or pointwise by  \eqref{dens def} are named Neumann fractional $\alpha$-stable random fields on $K$ of index $H=sd_w-\frac{(\alpha-1)d_h}{\alpha}$.   
\end{remark}

\begin{proof} \hfill

\begin{enumerate}
\item Let $f\in L^\alpha\left(K_w,\mu_{|_{K_w}}\right)$. Then, 
$$
\left\|W_{K_w,\alpha}(f)\right\|_{L^\alpha(K_w,\mu)}^{\alpha}=\int_{K_w}\left|f(y)\right|^{\alpha} \mu(dy)
$$
so that, by self-similarity of the Hausdorff measure $\mu$, 
$$
\left\|W_{K_w,\alpha}(f)\right\|_{L^\alpha(K_w,\mu)}^{\alpha}=\frac{1}{3^n} \int_{K}\left|f(F_w(z))\right|^{\alpha} \mu(dz).
$$
In particular, $f\circ F_w \in  L^\alpha(K,\mu)$, so that the stable random variable $W_{K,\alpha}(f\circ F_w)$ is well-defined. Moreover, 
$$
\left\|W_{K,\alpha}(f\circ F_\omega)\right\|_{L^\alpha(K_w,\mu)}^{\alpha}=\int_{K}\left|f(F_w(z))\right|^{\alpha} \mu(dz)=\left\|3^{n/\alpha} \,W_{K_w,\alpha}(f)\right\|_{L^\alpha(K,\mu)}^{\alpha}
$$
so that the two $S\alpha S$ random variables $W_{K,\alpha}(f\circ F_\omega)$ and $3^{n/\alpha} \,W_{K_w,\alpha}(f)$ have the same distribution, which establishes Assertion 1.
\item Let $f\in L^2_0\left(K_w,\mu_{|_{K_w}}\right)$. Then,   $f\circ F_w\in L^2_0(K,\mu)$  and 
$$\begin{array}{rcl}
\displaystyle \left(-\Delta\right)^{-s} \left(f\circ F_w\right) &=&\displaystyle \sum_{j=1}^{+\infty} \frac{1}{\lambda_j^s} \left(\int_{K}\Phi_j(y) f\left(F_w(y)\right)\mu (dy)\right) \Phi_j \\[15pt]
&=& \displaystyle 
5^{ns}3^{-n} \sum_{j=1}^{+\infty} \frac{1}{\lambda_{j,w}^s}\left( \int_{K}\Phi_j^w\left(F_w(y)\right) f\left(F_w(y)\right)\mu (dy)\right) \Phi_j^w\circ F_w.
\end{array}
$$
Hence, applying the change of variable $x=F_w(y)$ in the integrals and using the self-similarity of $\mu$, we get: 
$$\begin{array}{rcl}
\displaystyle \left(-\Delta\right)^{-s} \left(f\circ F_w\right) 
&=& \displaystyle 
5^{ns} \sum_{j=1}^{+\infty} \frac{1}{\lambda_{j,w}^s}\left( \int_{K_w}\Phi_j^w\left(x\right) f\left(x\right)\mu (dy)\right) \Phi_j^w\circ F_w
\\[20pt]
&=& \displaystyle  5^{ns} \left(\left(-\Delta_w\right)^{-s}f\right) \circ F_w.
\end{array}
$$
By definition of $H^{-s}(K)$, one deduces that if $f \in H^{-s}(K_w)$, then $f\circ F_w \in H^{-s}(K)$ and $$\left(-\Delta\right)^{-s} \left(f\circ F_w\right)=5^{ns} \left(\left(-\Delta_w\right)^{-s}f\right) \circ F_w.$$
Hence, for $f\in H^{-s}(K_w)$, 
$$
X_{s,\alpha}\left(f\circ F_w\right)=W_{K,\alpha}\left(\left(-\Delta\right)^{-s} \left(f\circ F_w\right)\right)=5^{ns}W_{K,\alpha}\left(\left(\left(-\Delta_w\right)^{-s}f\right) \circ F_w\right)
$$
and then by Assertion 1., 
$$
X_{s,\alpha}\left(f\circ F_w\right)\stackrel{(d)}{=} 5^{ns}3^{n/\alpha}W_{K_w,\alpha}\left(\left(-\Delta_w\right)^{-s}f\right)
$$
that is 
$$
X_{s,\alpha}\left(f\circ F_w\right)\stackrel{(d)}{=} 5^{ns}3^{n/\alpha}X_{s,\alpha}^w\left(f\right).
$$
Moreover, 
since $s=\frac{H}{d_w}+\frac{(\alpha-1)d_h}{\alpha d_w}= \frac{H\ln 2}{\ln 5}+\left(1-\frac{1}{\alpha}\right)\frac{\ln 3}{\ln 5},$
$$
5^{ns}=2^{nH}3^{n-n/\alpha}
$$
and so 
$$
X_{s,\alpha}\left(f\circ F_w\right)\stackrel{(d)}{=} 2^{nH}3^{n}X_{s,\alpha}^w\left(f\right)
.
$$
Assertion 2. then follows by linearity of $X_{s,\alpha}$ and of $X_{s,\alpha}^w$.
\item  For any $x\in K$, 
$$
\widetilde{X}_{s,\alpha}^w\left(F_w(x)\right)=\int_{K_w} G_s^w\left(F_w(x),y\right) W_{K_w,\alpha}(dy). 
$$
Hence, by linearity of the stochastic integrals with respect to $W_{K_w,\alpha}$ and $W_{K,\alpha}$ and by Assertion 1, 
$$3^{n/\alpha}
\left(\widetilde{X}_{s,\alpha}^w\left(F_w(x)\right)\right)_{x\in K}\stackrel{(fdd)}{=}  \left(  \int_{K} G_s^w\left(F_w(x),F_w(z)\right) W_{K,\alpha}(dz) \right)_{x\in K} .
$$
Moreover, according to the proof of Proposition 3.10 in \cite{BL22-FGF}, for any $x,y\in K$ such that $x\ne y$, 
\[
G_s^w(F_w(x),F_w(y)) =\frac{3^{n}}{5^{ns}} G_s(x,y).
\]
Therefore,   
$$
\frac{5^{ns}}{3^{n-n/\alpha}}\left(\widetilde{X}_{s,\alpha}^w\left(F_w(x)\right)\right)_{x\in K}\stackrel{(fdd)}{=}  \left(  \int_{K} G_s\left(x,z\right) W_{K,\alpha}(dz) \right)_{x\in K},
$$
that is 
$$
2^{nH}\left(\widetilde{X}_{s,\alpha}^w\left(F_w(x)\right)\right)_{x\in K}\stackrel{(fdd)}{=}  \left(  \widetilde{X}_{s,\alpha}(x)\right)_{x\in K},
$$
which concludes the proof.
\end{enumerate}
\end{proof}

\section{Dirichlet  Fractional Stable random fields }

The previous sections have focused on the construction of the fractional stable  random fields which were constructed using the Dirichlet form \eqref{dirichlet limit} and its associated Laplacian.
In this section we show how to perform a similar construction with Dirichlet boundary conditions at the 3 vertices of the Sierpi\'nski triangle. For fractional Gaussian fields this construction was done in \cite{BC22-DFGF}.

\subsection{Dirichlet fractional Riesz kernels}

Let us first recall the construction of the Laplacian with Dirichlet boundary conditions; we refer to \cite{BC22-DFGF} for further details. Let

\[
\mathcal F_D=\{f\in C(K):\lim_{m\to +\infty}\mathcal E_m(f,f)<\infty, f=0 \text{ on } V_0 \}.
\]

By Theorem 4.1 and Lemma 4.1 in \cite{FukushimaShima}, $(\mathcal E,\mathcal F_D)$ is a local regular Dirichlet form on $L^2(K,\mu)$. The generator of this Dirichlet form $\Delta_D$  is referred to as the Dirichlet Laplacian on $K$. 
 This Laplacian generates the Dirichlet heat semigroup $\{P_t^D\}_{t \ge 0}$ on $L^2(K,\mu)$ and the associated Dirichlet heat kernel  denoted by $p_t^D(x,y)$, for $t>0$ and $x,y \in K$, admits a uniformly convergent spectral expansion:
\begin{align}\label{spectralD}
p_t^D(x,y)=\sum_{j=1}^{+\infty} e^{-\lambda^D_j t} \Psi_j(x) \Psi_j(y)
\end{align}
where $0<\lambda^D_1\le \lambda^D_2\le  \cdots \le \lambda^D_j \le \cdots$ are the eigenvalues of $-\Delta_D$ and  $(\Psi_j)_{j\ge 1} \subset  \mathcal{D} (\Delta_D) \subset \mathcal F_D$ is an orthonormal basis of $L^2(K,\mu)$ such that  
\[
\Delta_D \Psi_j =-\lambda^D_j \Psi_j.
\]

From \cite{Lierl}, this heat kernel satisfies for some $c_{1},c_{2} \in(0,\infty)$,
\begin{equation}\label{eq:subGauss-upperD}
p_{t}^D(x,y)\leq c_{1} t^{-\frac{d_{h}}{d_{w}}}\!\exp\biggl(\!-c_{2}\Bigl(\frac{d(x,y)^{d_{w}}}{t}\Bigr)^{\frac{1}{d_{w}-1}}\!\biggr)
\end{equation}
for every \ $(x,y)\in K \times K$ and $t\in\bigl(0,+\infty)$ and  also satisfies the sub-Gaussian lower bound (see \cite{Lierl}):

\begin{equation}\label{eq:loweGauss-upperD}
p_{t}^D(x,y)\geq c_{1}\Psi_1(x) \Psi_1(y) t^{-\frac{d_{h}}{d_{w}}}\!\exp\biggl(\!-c_{2}\Bigl(\frac{d(x,y)^{d_{w}}}{t}\Bigr)^{\frac{1}{d_{w}-1}}\!\biggr)
\end{equation}
for every \ $(x,y)\in K \times K$ and $t\in\bigl(0,1)$ where  $\Psi_1$ is the first  eigenfunction of $-\Delta_D$. 

The main differences induced by the Dirichlet boundary conditions are the following:

\begin{itemize}
\item $\Psi_j(q)=0$ whenever $q \in V_0$;
    \item for all $x,y \in K$, $p_t^D(x,q)=p_t^D(q,y)=0$ whenever $q \in V_0$;
    \item uniformly on $K \times K$, $p_t^D(x,y) \to 0$ when $t \to +\infty$.
\end{itemize}

\begin{definition}
Let $s \ge 0$. For $f \in L^2(K,\mu)$, the Dirichlet fractional Laplacian $(-\Delta_D)^{-s}$ on $f$ is defined as
\[
(-\Delta_D)^{-s} f  =\sum_{j=1}^{+\infty} \frac{1}{(\lambda^D_j)^s} \Psi_j \int_K  \Psi_j(y) f(y) \mu(dy).
\]
\end{definition}

\begin{definition}
	For a parameter $s > 0$, we define the Dirichlet fractional Riesz kernel $G^D_s $ by 
	\begin{align}\label{green function Dirichlet}
	G^D_s(x,y)=\frac{1}{\Gamma(s)} \int_0^{+\infty} t^{s-1} p^D_t(x,y) dt, \quad x,y \in K, \, x\neq y.
	\end{align}
\end{definition}

For the Dirichlet boundary case, the analogue of the Schwartz space is given by
\[
\mathcal{S}_D(K)= \left\{ f \in C(K), f=0 \text{ on } V_0,   \forall k \ge 0 \lim_{n \to +\infty} n^k \left| \int_K  \Psi_n(y) f(y) \mu(dy) \right| =0 \right\}.
\]
We note that from \cite{BC22-DFGF}, for $s > 0$, $f \in \mathcal{S}_D(K)$, and $x \in K$
\begin{equation}\label{eq:Fractional Laplacian-Riesz Kernel Dirichlet}
 (-\Delta_D)^{-s} f (x) = \int_K G^D_{s}(x,y) f(y) \mu(dy).
\end{equation}

The Sobolev space $\mathcal{H}^{-s}_D (K)$ is defined as the completion of $\mathcal{S}_D(K)$ with respect to the inner product 
\[
\left\langle f ,g \right\rangle_{\mathcal{H}^{-s}_D (K)} =\int_K (-\Delta_D)^{-s} f (y) (-\Delta_D)^{-s} g (y) \mu (dy).
\]

The following result can be proved as in the case of the Neumann boundary condition and the details are let to the reader.

\begin{proposition}\label{estimate G Dirichlet}
	\
	\begin{enumerate}
		\item If  $s \in (0, d_h/d_w)$, there exist  constants $c,C >0$ such that for every  $x,y \in K$, $x \neq y$,
		\[
		c  \frac{ \Psi_1(x)\Psi_1(y)}{d(x,y)^{d_h-sd_w}} \le  G_s^D(x,y)  \le  \frac{C}{d(x,y)^{d_h-sd_w}}.
		\]
		\item If $s = d_h/d_w$, there exists a constant $C >0$ such that for every  $x,y \in K$, $x \neq y$
		\[
		c \Psi_1(x) \Psi_1(y) \max (1,| \ln d(x,y) |) \le G_s^D(x,y)  \le C \max (1, | \ln d(x,y) |).
		\]
	\end{enumerate}

 \end{proposition}
	\begin{lemma}\label{Integrabilite-G-Diri} Let $p \in (0,+\infty)$. Then for any $x\in K \setminus V_0$, $y \mapsto G_s(x,y)$ is in $L^p(K,\mu)$ iff $s>\max\left(\frac{(p-1)d_h}{p d_w},0\right)$. Moreover, for $s>\max\left(\frac{(p-1)d_h}{p d_w},0\right)$,
$$\sup_{x\in K} \int_{K} \left|G^D_s(x,y)\right|^p \mu(dy)<+\infty.
$$
\end{lemma}

\begin{proof}
The proof is similar to that of Lemma \ref{Integrabilite-G}. We however point out a slight difference which comes from the lower bounds in Proposition \ref{estimate G Dirichlet} involving $\Psi_1$. Assume $0<s<d_h/d_w$ and that $x \in K \setminus V_0$ is such that $y\to G_s(x,y) \in L^p(K,\mu)$. Then, from Proposition \ref{estimate G Dirichlet} we have  $\int_K \frac{ \Psi_1(y) }{d(x,y)^{p(d_h-sd_w)}} \mu(dy) <+\infty$. Let now $\varepsilon >0$ be small enough such that the closed ball $B(x,\varepsilon) \subset K\setminus V_0$, since $\psi_1$ is positive and continuous on $B(x,\varepsilon)$ we deduce that is bounded below by a positive constant on $B(x,\varepsilon)$. Therefore $\int_{B(x,\varepsilon)} \frac{ 1 }{d(x,y)^{p(d_h-sd_w)}} \mu(dy) <+\infty$ which yields $p(d_h-sd_w) <d_h$ using a proof similar to that of Lemma \ref{Integrabilite-dist}.
\end{proof}

\begin{theorem}\label{Holder-Riesz_Dirichlet}\hfill

\begin{enumerate}
\item  For any $s \in \left( \frac{d_h}{d_w} , 1\right)$, there exists a constant $C>0$ such  that for every  $x,y,z \in K$,

\[
| G^D_s (x,z)-G^D_s(y,z)|  \le Cd(x,y)^{ sd_w-d_h} .
\]
\item  For $s \ge 1$, there exists a constant $C>0$ such  that for every  $x,y,z \in K$,

\[
| G^D_s (x,z)-G^D_s(y,z)|  \le Cd(x,y)^{ d_w-d_h} \max\left(\left|\ln d(x,y)\right|,1 \right).
\]

\end{enumerate}
\end{theorem}

\subsection{Dirichlet fractional stable fields}

\begin{definition}\label{def fractional field diri} 
  Let $s >0$ and $\alpha \in (0,2]$. The  Dirichlet fractional $\alpha$-stable field $\left\{ X_{s,\alpha}^D(f), f\in {\mathcal{H}^{-s}_D (K)} \right\}$ is defined as
  \[
X_{s,\alpha}^D(f) =W_{K,\alpha} ((-\Delta_D)^{-s} f)=\int_K (-\Delta_D)^{-s} f (z) \,W_{K,\alpha}(dz).
  \]
\end{definition}
\begin{definition} Let $\alpha \in (0,2]$ and $s>\max\left(\frac{(\alpha-1)d_h}{\alpha d_w},0\right) $. Then for any $x\in K$,  by Lemma \ref{Integrabilite-G}, the Dirichlet Riesz kernel $z\mapsto G^D_s(x,z)$ is in $L^{\alpha}(K,\mu)$ and so a fractional $\alpha$-stable field $\widetilde{X}^D_{s,\alpha}=\left(\widetilde{X}_{s,\alpha}^D(x)\right)_{x\in K}$ 
is well-defined by 
$$
\widetilde{X}_{s,\alpha}^D(x)=\int_K  G^D_s(x,z) W_{K,\alpha}(dz), \ x\in K.
$$
This field, defined pointwise, is called the Dirichlet fractional $\alpha$-stable of order $H=sd_\omega-\frac{(\alpha-1)d_h}{\alpha }$. 
\end{definition}

Note that since $G_s(q,z)=0$ whenever $q \in V_0$, as expected we get that the Dirichlet fractional  $\alpha$-stable field vanishes at the 3 vertices of the Sierpi\'nski triangle.

The following theorem can be proved as Theorem \ref{Existence-Density}.

\begin{theorem}\label{Existence-Density Dirichlet}
  Let $\alpha\in (0,2)$. 
  
  \begin{itemize}
\item   If    $s>\max\left(\frac{(\alpha-1)d_h}{\alpha d_w},0\right)$, then 
for every $ f\in \mathcal{S}_D(K)$, we have a.s.
  \[
\int_K f(x) \widetilde{X}_{s,\alpha}^D(x) \mu(dx) =X_{s,\alpha}^D(f)
  \]
  with $\widetilde{X}_{s,\alpha}^D=\left(\widetilde{X}_{s,\alpha}^D(x)\right)_{x\in K}$ the fractional $\alpha$-stable field  well-defined by 
\begin{equation}
\label{dens def diri}
\widetilde{X}_{s,\alpha}^D(x)=\int_K  G^D_s(x,z) W_{K,\alpha}(dz), \ x\in K. 
\end{equation}
\item Moreover, if there exists a  random field $(Y(x))_{x \in K}$ such that a.s
  \[
\int_K |Y(x)| \mu(dx) <+\infty 
  \]
 and such that for every $ f\in \mathcal{S}(K)$ we have a.s.
 \[
\int_K f(x) Y(x) \mu(dx) =X_{s,\alpha}^D(f),
  \]
  then  we have $s>\max\left(\frac{(\alpha-1)d_h}{\alpha d_w},0\right)$ and for a.e. $x\in K$  with probability one $Y (x)=\widetilde{X}_{s,\alpha}^D(x)$.
\end{itemize}
\end{theorem}

The following result can then be proved as Theorem \ref{SPR} and Theorem \ref{SPR alpha >1}. 

\begin{theorem}
\label{SPR2}
 Let $\alpha\in (0,2)$, $s>\max\left(\frac{(\alpha-1)d_h}{\alpha d_w},0\right)$ and let  $\widetilde{X}_{s,\alpha}^D$ be the Dirichlet fractional $\alpha$-stable field. Let us set $\eta_s=\min(s,1)d_w-d_h$ and 
$$
\beta_s=\left\{\begin{array}{ll}
0&\textrm{if $s< 1$}\\
1&\textrm{if $s\ge 1$.}
\end{array}
\right. $$

 \begin{enumerate}
\item If  $s>\frac{d_h}{d_w}$ and if $\alpha \ge 1$, 
 then there exists a modification $\widetilde{X}_{s,\alpha}^{D,*}$ of $\widetilde{X}_{s,\alpha}^D$ such that  
	\begin{equation}
	\label{ModContX2}
	\lim_{\delta \to 0}\underset{\underset{x,y\in K}{\scriptstyle 0<d(x,y)}\le \delta}{\sup}\  \frac{\left| \widetilde{X}_{s,\alpha}^{D,*}(x)-\widetilde{X}_{s,\alpha}^{D,*}(y)\right|}{d(x,y)^{\eta_s} |\ln d(x,y)|^{\beta_s+\frac{1}{2}}} <\infty.	
	\end{equation} 
 \item If  $s>\frac{d_h}{d_w}$ and if $\alpha <1$, 
 then there exists a modification $\widetilde{X}_{s,\alpha}^{D,*}$ of $\widetilde{X}_{s,\alpha}^D$ such that  
	\begin{equation}
	\label{ModContX2-alpha>1}
	\lim_{\delta \to 0}\underset{\underset{x,y\in K}{\scriptstyle 0<d(x,y)}\le \delta}{\sup}\  \frac{\left| \widetilde{X}_{s,\alpha}^{D,*}(x)-\widetilde{X}_{s,\alpha}^{D,*}(y)\right|}{d(x,y)^{\eta_s} |\ln d(x,y)|^{\beta_s}} <\infty.	
	\end{equation} 
 \item If $s\le \frac{d_h}{d_w}$, a.s., the sample paths of $\widetilde{X}_{s,\alpha}$ are unbounded on the compact $K$. 
 \end{enumerate}
\end{theorem}

Finally, we conclude by pointing out that the Dirichlet fractional  $\alpha$-stable field on the Sierpi\'nski  gasket satisfies the same invariance properties as for the Neumann field which are stated in Section \ref{invariance fields}. 

\appendix
\section{Proof of Lemma \ref{conv-G}}
\label{conv-G-proof}

We first observe that according to Proposition \ref{Control-G} and Lemma \ref{Integrabilite-dist} one has for every $x \in K$
\begin{align}\label{integrability G}
\int_K   \int_K \left|G_s(x,u)G_{t}(u,y)\right|\mu(du)   \mu (dy)<+\infty.
\end{align}
Now, on one hand, from Lemma \ref{riesz-kernel laplace}, we have for every $f \in \mathcal S(K)$
\[
(-\Delta)^{-s-t} f (x)=\int_K G_{s+t}(x,y) f(y) \mu(dy).
\]

On the other hand, from  \eqref{composition laplace frac},  Lemma \ref{riesz-kernel laplace} and Fubini theorem which is justified thanks to \eqref{integrability G}, we have
\begin{align*}
(-\Delta)^{-s-t} f (x) & =\int_K G_{s}(x,y) (-\Delta )^{-t}f(y) \mu(dy) \\
 &=\int_K G_{s}(x,y) \left(\int_K G_t (y,u) f(u)\mu(du)\right) \mu(dy) \\
 &=\int_K \left(  \int_K G_s(x,u)G_{t}(u,y)\mu(du) \right) f(y) \mu (dy).
\end{align*}
This gives that for every $f \in \mathcal S(K)$,
\[
\int_K \left(  \int_K G_s(x,u)G_{t}(u,y)\mu(du) \right) f(y) \mu (dy)=\int_K G_{s+t}(x,y) f(y) \mu(dy).
\]
 We use this equality with $f=p_t(z,\cdot)-1 \in\mathcal{S}(K)$ and let $t \to 0$. Since for every function $g \in L^1(K,\mu)$, $z \to \int_K p_t(z,y) g(y) \mu(dy)  $ converges in $L^1$ to $g$ when $t \to 0$ (see \cite[Page 114]{BauBook}),
 this gives that for $x \in K$ and $\mu$ a.e. $z \in K$
 \[
 G_{s+t}(x,z)=\int_K G_s(x,u)G_{t}(u,z)\mu(du).
 \]

 The conclusion will then follow from the following lemma.

 \begin{lemma}\hfill

 \begin{enumerate}
 \item  Let $s>0$. The map $(x,y) \to G_s(x,y)$  is continuous on the set $\{ (x,y) \in K^2, x \neq y \} $. If moreover $s > d_h/d_w$, it is continuous on $K \times K$. 
  \item Let $s,t>0$ and $x\in K$. The map $z \to \int_K G_s(x,u)G_{t}(u,z)\mu(du)$  is continuous on the set $\{ z \in K, z \neq x \} $. If moreover $s+t> d_h/d_w$, then it is continuous on $K$.
  \end{enumerate}
 \end{lemma}

 \begin{proof}
 We recall that
 \begin{align}\label{def g appendix}
G_s(x,y)=\frac{1}{\Gamma(s)} \int_0^{+\infty} t^{s-1} (p_t(x,y) -1)dt, \quad x,y \in K, \, x\neq y
 \end{align}
 and that the map $(t,x,y) \to p_t(x,y)$ is continuous on $(0,+\infty) \times K \times K$. 
 
 It follows from  Lemma 2.3 in \cite{BL22-FGF} that there exists a constant $C>0$ such that for $t \ge 1$ and $x,y\in K$, $|p_t(x,y)-1| \le C e^{-\lambda_1 t}$. On the other hand, if $\varepsilon >0$, $0<t <1$, $d(x,y) \ge \varepsilon$, then from \eqref{eq:subGauss-upper} one has
 \[
| p_t(x,y) -1|  \leq c_{3}t^{-d_{h}/d_{w}}\exp\biggl(-c_{4}\Bigl(\frac{\varepsilon^{d_{w}}}{t}\Bigr)^{\frac{1}{d_{w}-1}}\biggr)+1.
 \]
 It then easily follows from the representation \eqref{def g appendix} that $(x,y) \to G_s(x,y)$  is continuous on the set $\{ (x,y) \in K^2, d(x,y) > \varepsilon \} $ and the conclusion of Assertion 1 follows since $\varepsilon$ is arbitrary. To prove Assertion 1, for $s>d_h/d_w$, one can simply use the upper bound $p_t(x,y) \le c_{3}t^{-d_{h}/d_{w}}$ which is valid for all $t \in (0,1)$, $x,y \in K.$ The proof of Assertion 2 relies on Assertion 1 and the upper bounds in Proposition \ref{Control-G}. The details are left to the reader.
 \end{proof}

\bibliographystyle{plain}

\noindent
\textbf{Fabrice Baudoin}: \url{fbaudoin@math.au.dk}\\
Department of Mathematics,
Aarhus University,
Denmark

\medskip

\noindent
\textbf{Céline Lacaux}: \url{celine.lacaux@univ-avignon.fr}\\
LMA Université d'Avignon, 
UPR 2151, 
84029, Avignon, France
\end{document}